\documentclass[11pt,a4paper,reqno]{amsart}
\usepackage{amsthm,amsmath,amsfonts,amssymb,amsxtra,appendix,bookmark,dsfont,latexsym,bm,hyperref,delarray,color,euscript,amsgen,amsbsy,amsopn,amscd,latexsym,mathrsfs}

\makeatletter

\usepackage{hyperref}

\makeatletter
\def\@tocline#1#2#3#4#5#6#7{\relax
  \ifnum #1>\c@tocdepth 
  \else
    \par \addpenalty\@secpenalty\addvspace{#2}%
    \begingroup \hyphenpenalty\@M
    \@ifempty{#4}{%
      \@tempdima\csname r@tocindent\number#1\endcsname\relax
    }{%
      \@tempdima#4\relax
    }%
    \parindent\z@ \leftskip#3\relax \advance\leftskip\@tempdima\relax
    \rightskip\@pnumwidth plus4em \parfillskip-\@pnumwidth
    #5\leavevmode\hskip-\@tempdima
      \ifcase #1
       \or\or \hskip 1em \or \hskip 2em \else \hskip 3em \fi%
      #6\nobreak\relax
      \dotfill
      \hbox to\@pnumwidth{\@tocpagenum{#7}}
    \par
    \nobreak
    \endgroup
  \fi}
\makeatother

\newtheorem{theorem}{Theorem}
\newtheorem{lemma}[theorem]{Lemma}
\newtheorem{proposition}[theorem]{Proposition}

\theoremstyle{definition}
\newtheorem{definition}{Definition}

\theoremstyle{remark}
\newtheorem{remark}[theorem]{Remark}

\newcommand\R{{\ensuremath {\mathbb R} }}
\newcommand\C{{\ensuremath {\mathbb C} }}
\newcommand\N{{\ensuremath {\mathbb N} }}

\newcommand\1{{\ensuremath {\mathds 1} }}

\newcommand\nn{\nonumber}

\newcommand{\bS}{\mathbb{S}}

\newcommand{\wto}{\rightharpoonup}

\newcommand{\cQ}{\mathcal{Q}}

\renewcommand{\epsilon}{\varepsilon}

\renewcommand{\ge}{\geqslant}
\renewcommand{\le}{\leqslant}
\renewcommand{\geq}{\geqslant}
\renewcommand{\leq}{\leqslant}

\renewcommand{\tilde}{\widetilde}

\newcommand{\eps}{\varepsilon}

\numberwithin{equation}{section}

\begin{document}

\title[Focusing NLS with Hardy potential]{Uniqueness of ground state and minimal-mass blow-up solutions for focusing NLS with Hardy potential}

\author[D. Mukherjee]{Debangana Mukherjee}
\address{Department of Mathematics and Statistics, Masaryk University, Brno, Czech Republic} 
\email{mukherjeed@math.muni.cz}

\author[P.T. Nam]{Phan Th\`anh Nam}
\address{Department of Mathematics, LMU Munich, Theresienstrasse 39, D-80333 Munich, and Munich Center for Quantum Science and Technology (MCQST), Schellingstr. 4, D-80799 Munich, Germany} 
\email{nam@math.lmu.de}

\author[P.T. Nguyen]{Phuoc-Tai Nguyen}
\address{Department of Mathematics and Statistics, Masaryk University, Brno, Czech Republic}
\email{ptnguyen@math.muni.cz}

\date{\today}

\begin{abstract} We consider the focusing nonlinear Schr\"odinger  equation with the critical inverse square potential. We give the first proof of  the uniqueness of the ground state solution. Consequently, we obtain a sharp Hardy-Gagliardo-Nirenberg interpolation inequality. Moreover, we provide a complete characterization for the minimal mass blow-up solutions to the time dependent problem. 

\medskip
	
	
\medskip
	

\end{abstract}

\maketitle

\tableofcontents

\section{Introduction}

In this paper, we consider the Cauchy problem for the focusing nonlinear Schr\"odinger (NLS) equation with the critical inverse square potential 
\begin{equation} \label{eq:NLS} 
\left\{  \begin{aligned}
i \partial_t u(t,x) &= (-\Delta -c_*|x|^{-2}) u(t,x) - |u(t,x)|^{p-2}u(t,x), \quad x\in \R^d, t>0, \\
u(0,x) &=u_0(x), \quad x\in \R^d.
\end{aligned} \right. \end{equation}
Here 
$$d\geq 3,\quad 2<p<2^*=\frac{2d}{d-2} \quad \text{and}\quad c_* = \frac{(d-2)^2}{4}. $$
Since $c_*$ is the best constant in Hardy's inequality
$$
\int_{\R^d} |\nabla \phi(x)|^2 dx \ge c_* \int_{\R^d}\frac{|\phi(x)|^2}{|x|^{2}}dx, \quad \forall \phi\in H^1(\R^d), 
$$ 
the term $-c_*|x|^{-2}$ is thus called the Hardy potential.

The classical focusing NLS equation (without Hardy potential) is a huge subject; see e.g. \cite{Cazenave-03,Tao-06} for excellent textbooks.  In the last decades, a considerable amount of work has been devoted to the study of the NLS equation with an inverse square potential. Since the singularity of $|x|^{-2}$ is critical to the Laplacian (as we can see from Hardy's inequality), its effect cannot be simply  understood by perturbation methods. Therefore, several well-known results for the classical NLS are not standardly extended to this case. 

The scale-covariance operator $H_c =\Delta -c |x|^{-2}$ plays an important role in quantum mechanics \cite{Kalf-Schmincke-Walter-Wuest-75}. The heat equation associated with $H_c$ has been first studied by Vazquez-Zuazua \cite{Vazquez-Zuazua-2000}. In the subcritical case $c<c_*$, the Strichartz's estimates for $H_c$ have been established by Burq-Planchon-Stalker-Tahvildar-Zadeh \cite{BurPlaStaTah-03}, using Rodnianski-Schlag's approach \cite{RodSch-04} (originally developed  for potentials decays like $|x|^{-2-\eps}$ for large $|x|$) but bypassing some dispersive estimates thanks to Kato-Yajima smoothing and  Morawetz estimates. In particular,  the key results in \cite{BurPlaStaTah-03} ensure the local well-posedness of the corresponding NLS  from standard techniques. For further results in the subcritical case $c<c_*$, see  Zhang-Zheng \cite{ZhaZhe-14}, Killip-Murphy-Visan-Zheng \cite{KilMurVisZhe-17} and  Lu-Miao-Murphy \cite{LuMiMu-18} for the long-time behavior of solutions; Killip-Miao-Visan-Zhang-Zheng \cite{KilMiaVisZhaZhe-17} for the global well-posedness and scattering in the energy-critical case; Csobo-Genoud \cite{CsoGen-18} 
for the classification of the mass-critical blow-up solutions; and Bensouilah-Dinh-Zhu \cite{BenDinZhu-18} for the stability and instability of standing waves. 

The critical case $c=c_*$ is more interesting, but less understood. In this case, the analysis is more subtle and intricate in several aspects, for example the energy space of the NLS \eqref{eq:NLS} is a proper subspace of $H^1(\R^d)$ and the endpoint Strichartz's estimates fail \cite{BurPlaStaTah-03}. Nevertheless, the non-endpoint  Strichartz's estimates for \eqref{eq:NLS} remain valid, as proved by Suzuki \cite{Suz-16}, and hence the local well-posedness in the energy-subcritical $p<2^*$ follows from the abstract framework in  Okazawa-Suzuki-Yokota \cite{OkaSuzYok-12,OkaSuzYok-12a}. In another development, the existence and the asymptotic behavior of the standing waves of the NLS \eqref{eq:NLS} have been established by Trachanas-Zographopoulos \cite{TraZog-15}. 

In spite of the above remarkable works, there are still several open questions regarding the NLS with an inverse square potential. In the present paper, we will  prove the uniqueness of the ground state solution, thus extending the fundamental results of Coffman \cite{Coffman-72} and McLeod-Serrin \cite{LeoSer-87}  for the classical NLS. As a consequence, we also obtain the full characterization for the minimal mass blow-up solution for \eqref{eq:NLS}, in the spirit of Merle \cite{Mer-93}. 

To simplify the representation, we will mostly focus on the critical case $c=c_*$, but our results can be extended easily to the subcritical case $c<c_*$. The precise statements of our results are represented in the next section.

\subsection*{Acknowledgements.} Debangana Mukherjee and Phuoc-Tai Nguyen are supported by Czech Science Foundation, project GJ19 -- 14413Y. Phan Th\`anh Nam is funded by the Deutsche Forschungsgemeinschaft (DFG, German Research Foundation) under Germany's Excellence Strategy -- EXC-2111 -- 390814868.

\section{Main results}

\subsection{Ground state problem}

Let us start by recalling some fundamental facts. By Hardy's inequality
$$
\int_{\R^d} |\nabla u(x)|^2 dx -c_* \int_{\R^d} \frac{|u(x)|^2}{|x|^2}dx \geq 0,\quad \forall u\in H^1(\R^d)
$$
and Friedrichs' method, we may extend the operator 
$$H=-\Delta-c_*|x|^{-2}$$
to be a non-negative self-adjoint operator on $L^2(\R^d)$. 

The energy space associated with \eqref{eq:NLS} is the quadratic form domain $\cQ$ of $H$, which is the Hilbert space with norm
$$
\|u\|_{\cQ} =\left( \|\sqrt{H}u\|_{L^2(\R^d)}^2  + \|u\|^2_{L^2(\R^d)}  \right)^{\frac{1}{2}}.
$$
Obviously
$$
H^1(\R^d) \subset \cQ\subset L^2(\R^d).
$$
Moreover, $H^1(\R^d)$ is strictly included in $\cQ$. In fact, it is possible to construct a function $\varphi \in \cQ$ such that $
\varphi(x) \sim |x|^{-\frac{d-2}{2}}$ as $|x|\to 0$ \footnote{Note that the function $\psi(x)=|x|^{-\frac{d-2}{2}}$ is exactly the ``ground state" of Hardy's inequality, namely $(-\Delta -c_*|x|^{-2}) \psi=0$ for all $x\ne 0$, but $\psi\not\in L^2(\R^2)$ and hence Hardy's inequality is strict.}, and hence 
$$\cQ \not\subset L^{2^*}(\R^d),$$
while $H^1(\R^d)\subset L^{2^*}(\R^d)$ by Sobolev's embedding.  On the other hand, a fundamental result of Brezis-V\'azquez \cite[Theorem 4.1]{BreVaz-97} ensures the continuous embedding 
\begin{equation} \label{eq:Q-in-Lp}
\cQ \subset L^{p}(\R^d), \quad \forall \, p \in [2,2^*).
\end{equation}
Actually, there is a stronger result proved by Frank \cite[Theorem 1.2]{Fra-09}  that $\cQ$ can be continuously embedded into fractional Sobolev spaces: 
\begin{equation} \label{eq:Q-in-Hs}
\cQ \subset H^s(\R^d), \quad \forall s \in (0,1).
\end{equation}
See also Solovej-S\o rensen-Spitzer \cite{SSS-10}  for a previous result on relativistic particles and Vazquez-Zuazua \cite{Vazquez-Zuazua-2000} for a  local version of \eqref{eq:Q-in-Hs}. Consequently, for every $2<p<2^*$ and $\varphi \in \cQ$ the following energy functional is well-defined
$$
E(\varphi):=\frac{1}{2} \|\sqrt{H}\varphi\|_{L^2(\R^d)}^2 - \frac{1}{p} \|\varphi\|_{L^p(\R^d)}^p. 
$$

The Euler-Lagrange equation associated to the minimizer of this energy functional under the mass constraint (namely $\|\varphi\|_{L^2(\R^d)}$ is fixed) reads
\begin{align} \label{eq:GS-mu}
H Q - |Q|^{p-2} Q - \mu Q =0
\end{align}
with a constant  $\mu\in \R$ (the Lagrange multiplier or chemical potential). Note that any solution of \eqref{eq:GS-mu} gives a solution of the NLS \eqref{eq:NLS} of the form
$$
u(t,x) = e^{-i\mu t} Q(x). 
$$

By a standard scaling argument (i.e. replacing $u(x)$ by $a u(bx)$ with constants $a,b>0$) from now on we will fix the Lagrange multiplier $\mu=-1$ for simplicity. 

Our first main result is the existence and uniqueness of a radial positive solution to the ground state equation. 

\begin{theorem}[Existence and uniqueness of the ground state solution] \label{thm:ground-state} Let $d\ge 3$ and $2<p<2^*$. Then there exists a unique radial positive solution $Q\in \cQ$ to the equation
\begin{align} \label{eq:GS}
H Q - Q^{p-1} + Q =0.
\end{align}
\end{theorem}

As a consequence of Theorem \ref{thm:ground-state}, we obtain a sharp interpolation inequality. 

\begin{theorem}[Hardy-Gagliardo-Nirenberg inequality] \label{thm:HGN} Let $d\ge 3$ and $2<p<2^*$. Then we have 
\begin{equation} \label{eq:HGN}
\|\sqrt{H}u\|_{L^2(\R^d)}^{\theta} \|u\|_{L^2(\R^d)}^{1-\theta} \ge C_{\rm HGN} \|u\|_{L^p(\R^d)}, \quad \theta=\frac{d}{2}-\frac{d}{p}, 
\end{equation}
for all $u\in \cQ$, with the sharp constant  
\begin{equation} \label{eq:CHGN}
C_{\rm HGN} = \|Q\|_{L^2(\R^d)}^{\frac{p-2}{p}} (1-\theta)^{\frac{1}{p}} \left( \frac{\theta}{1-\theta} \right)^{\frac{d(p-2)}{4p}} . 
\end{equation}
Moreover, any optimizer of the interpolation inequality \eqref{eq:HGN} has the form 
$$u(x)=z Q(\lambda x)$$  with constants $z\in \mathbb{C}$ and $\lambda>0$. Here $Q$ is the unique solution in Theorem \ref{thm:ground-state}. 
\end{theorem}

While the existence of solution to \eqref{eq:GS} is well-known (it essentially follows from Weinstein's method \cite{Wei-83}), the uniqueness part  has remained an open problem for a while. Our result in Theorem \ref{thm:ground-state} can be extended to the subcritical case $0<c<c_*$ where it is also new (see Section \ref{sec:ext} for further discussions). 

For the classical NLS (without the inverse square potential), the uniqueness of the ground state solutions goes back to Coffman \cite{Coffman-72} who treated the cubic NLS ($p=4$) in 3D and McLeod-Serrin \cite{LeoSer-87} who handled the general case. It is worth noting that the radial symmetry of the ground state solution generally follows from rearrangement inequalities (see e.g. \cite{LieLos-01,Burchard-09}), although the symmetry can be also derived from the equation itself by the moving plan method (see Gidas-Ni-Nirenberg \cite{GidNiNir-81} and Kwong \cite{Kwong-89}). However, it seems that the existing methods for the classical NLS do not apply easily to \eqref{eq:GS} due to the strong effect of the critical Hardy potential. 

Our ideas of proving the uniqueness in Theorem \ref{thm:ground-state} are as follows. First, since the solution behaves as $|x|^{-(d-2)/2}$ close to the origin (similarly to the non-integrable ground state of Hardy's inequality), we introduce the function 
$$ 
v(x)=|x|^{\frac{d-2}{2}}Q(x)
$$
which is uniformly bounded and decays fast at infinity. The equation for $v$ admits a generalized Pohozaev type identity, following the spirit of the recent work of Shioji-Watanabe \cite{ShiWat-13} (although the uniqueness result in \cite{ShiWat-13} does not apply in our case  as the condition (II) in  \cite[Theorem 1]{ShiWat-13} is so restricted). Then the conclusion follows from  a careful implementation of the general shooting argument of Yanagida \cite{Yanagida-91}, taking the specific scaling of equation \eqref{eq:GS} into account.


We will prove Theorems \ref{thm:ground-state} and \ref{thm:HGN} in Section \ref{sec:HGN}.

\subsection{Dynamical problem}

Next, we consider the NLS \eqref{eq:NLS} with an initial datum $u_0\in \cQ$.

\begin{definition}[Weak solutions] A function $u$ is called a {\em weak solution} to \eqref{eq:NLS} in $(0,T)$ with initial datum 
$u_0\in \cQ$  if 
$$u \in C([0,T);\cQ) \cap C^1([0,T);\cQ^*)$$
and it satisfies the Duhamel formula
\begin{align} \label{eq:Duhamelu} 
u(t)=e^{-itH}u_0 + i\int_0^t e^{-i(t-s)H}|u(s)|^{p-2}u(s)ds \quad \forall t \in (0,T).
\end{align}
Here we ignore the $x$-dependence in the notation. If $T=\infty$ then $u$ is called a {\em global solution}. 
\end{definition}

For the reader's convenience, we collect some basic  properties of equation \eqref{eq:NLS} with initial datum in the energy space $\cQ$. Recall that $Q$ is the unique ground state solution of \eqref{eq:GS}.

\begin{theorem}[Basic properties of focusing NLS with Hardy potential] \label{thm:NLS-basic} Let $d \geq 3$ and $2<p<2^*$. For any $u_0\in \cQ$ the followings hold true.
\begin{itemize}
\item[(i)] {\em (Local well-posedness).} There exist a constant $T=T(\|u_0\|_\cQ)>0$ and a unique  weak solution $u \in C([0,T);\cQ) \cap C^1([0,T);\cQ^*)$ of \eqref{eq:NLS} in $(0,T)$. The mass and energy conservation laws hold, i.e. 
\begin{align} \label{eq:conservation}
\| u(t) \|_{L^2(\R^d)} = \| u_0 \|_{L^2(\R^d)}, \quad E(u(t)) = E(u_0), \quad \forall t \in [0,T),
\end{align}
Moreover, the above solution $u$ admits a unique continuation up to a maximum time $T^*$ such that either the solution is global, namely $T^*=\infty$, or the solution blows up at finite time, namely $T^*<\infty$ and
\begin{equation} \label{blowup} 
\lim_{t \nearrow T^*}\| \sqrt{H} u \|_{L^2(\R^d)}=\infty.
\end{equation}
\item[(ii)] {\em (Sufficient conditions for global existence).} The solution of \eqref{eq:NLS} is global in one of the following three cases: $2<p<2+4/d$; $p=2+4/d$ and $\|u_0\|_{L^2(\R^d)}<\|Q\|_{L^2(\R^d)}$; $2+4/d<p<2^*$ and
$$
E(u_0) \|u_0\|_{L^2}^{q} < E(Q) \|Q\|_{L^2}^{q}, \quad \| \sqrt{H} u_0\|_{L^2}^2 \|u_0\|_{L^2}^{q} < \| \sqrt{H} Q\|_{L^2}^2 \|Q\|_{L^2}^{q}
$$
with 
\begin{equation} \label{q}
q=\frac{4d + 4p - 2pd}{dp-2d-4}.
\end{equation}
Moreover, in any of the above cases, $\| \sqrt{H}u(t) \|_{L^2(\R^d)}$ remains uniformly bounded in $t \in (0,\infty)$.
\item[(iii)] {\em (Sufficient conditions for blowup).} The solution of \eqref{eq:NLS} blows up at finite time if $2+4/d \le p <2^*$, $|x| u_0 \in L^2(\R^d)$ and either $E(u_0)<0$, or  
$$
E(u_0) \|u_0\|_{L^2}^{q} < E(Q) \|Q\|_{L^2}^{q}, \quad \| \sqrt{H} u_0\|_{L^2}^2 \|u_0\|_{L^2}^{q} > \| \sqrt{H} Q\|_{L^2}^2 \|Q\|_{L^2}^{q}
$$
with $q$ as in \eqref{q}. 
\end{itemize}
\end{theorem}

Most of the results in Theorem \ref{thm:NLS-basic} are known or easily obtained by adapting the analysis for the usual NLS. More precisely, the local well-posedness is due to Suzuki \cite{Suz-16}, based on the abstract result of Okazawa-Suzuki-Yokota \cite{OkaSuzYok-12,OkaSuzYok-12a} (the uniqueness in (i) was not addressed  explicitly but it follows from the standard method in \cite{Cazenave-03} using the Strichartz estimates for Hardy potential by Suzuki \cite[Proposition 4.8]{Suz-16}; see also Mizutani \cite[Corollary 2.3]{Miz-17}). In the global existence (ii), the first case (of subcritical nonlinearity) is standard; the second case is analogous to Weinstein's famous theorem for the usual NLS \cite{Wei-83}; the last case follows from the strategy in the usual NLS of Kenig-Merle \cite{KenMer-08} (we will recall the representation of Holmer-Roudenko \cite{HolRou-07} for the energy subcritical case, see also Ogawa-Tsutsumi \cite{OgaTsu-91} for an earlier work on the radial case). In the blow-up conditions (iii), the first case (of negative energy) can be found in Suzuki \cite[Theorem 4.1]{Suz-17}, while the other case follows from the analysis for the usual NLS \cite{KenMer-08,HolRou-07}. 

For the sake of completeness, in Section \ref{sec:basic} we will briefly recall standard arguments in the proof of Theorem \ref{thm:NLS-basic}. Some key ingredients, e.g. Virial's identities, will be useful later. 

\bigskip

Now we concentrate on the mass-critical case $p=2+4/d$. Our new result is a complete characterization of the minimal-mass blow-up solutions. Recall that $Q$ is the unique ground state solution of \eqref{eq:GS}. 

\begin{theorem}[Minimal mass blow-up solutions] \label{thm:minimal-mass-solution} Assume $d \geq 3$ and $p=2+4/d$. 
\begin{itemize}
\item[(i)] {\em (Existence).} Let $\gamma\in \mathbb{R}$, $\lambda>0$ and $T>0$ arbitrarily. Then 
\begin{equation} \label{eq:u-0-T}
u(t,x)= e^{i \gamma}e^{i\frac{\lambda^2}{T-t}}e^{-i\frac{|x|^2}{4(T-t)}}\left( \frac{\lambda}{T-t} \right)^{\frac{d}{2}} Q\left( \frac{\lambda x}{T-t} \right) \quad \forall x \in \R^d, \; t\in [0,T)
\end{equation}
solves the NLS \eqref{eq:NLS} in $(0,T)$ and blows up at the finite time $T$. 

\item[(ii)] {\em (Uniqueness).} For any finite time $T>0$, if $u \in C([0,T);\cQ)\cap C^1([0,T),\cQ^*)$ is solution of \eqref{eq:NLS} in $(0,T)$ and blows up at $T$, then $u$ is given in \eqref{eq:u-0-T} for some constants  $\gamma\in \mathbb{R}$ and $\lambda>0$. 
\end{itemize}
\end{theorem}	

Theorem \ref{thm:minimal-mass-solution} is an extension of Merle's celebrated  result \cite{Mer-93} (for the classical NLS) to the case of Hardy potential. A similar result  was obtained recently  by Csobo-Genoud  \cite{CsoGen-18} in the subcritical case  $c<c_*$ (although the uniqueness result for the ground state was not available in \cite{CsoGen-18}). 

A crucial ingredient of the proof of Theorem \ref{thm:minimal-mass-solution} is a compactness lemma on the minimizing sequences of the Hardy-Gagliardo-Nirenberg inequality \eqref{eq:HGN}. In the classical NLS, the compactness lemma follows standardly from the concentration-compactness method. This method was also used in \cite{CsoGen-18} for the subcritical case  $c<c_*$, but the analysis cannot be extended to the case $c=c_*$. In the present paper, we will prove the compactness result for the critical case by a refined method based on geometric localization techniques. This enables us to implement the approach of Hmidi-Keraani \cite{HmiKer-05} to achieve the full characterization for the mass-critical blow-up solutions.

Finally, let us remark that all the results in Theorems \ref{thm:ground-state}, \ref{thm:HGN}, \ref{thm:NLS-basic}, \ref{thm:minimal-mass-solution} hold true in the subcritical case, where the Hardy potential $-c_*|x|^{-2}$ is replaced by $-c|x|^{-2}$ with $c<c_*$. In fact, the proof in the subcritical case is similar, even simpler, as will be explained in Section 6. 

\bigskip

\noindent \textbf{Organization of the paper.} Section 3 is devoted to the proof of Theorem \ref{thm:ground-state} and Theorem \ref{thm:HGN}. In Section 4, we present basic properties of problem \eqref{eq:NLS}. In Section 5, we demonstrate the characterization of blow-up solutions with minimal mass. Finally, in Section 6, we discuss the extension to the subcritical case $0<c<c_*$.  

\bigskip

\noindent \textbf{Notations.} 
\begin{itemize}
\item 
When there is no confusion, we will write $u(t)$ instead of $u(t,x)$, $||u||_{L^p}$ instead of $||u||_{L^p(\R^d)}$. 
\item For $u,v \in \cQ$, we use the notation
$$ \langle u, Hv \rangle = \int_{\R^d} u Hv dx = \int_{\R^d} \left( \nabla u \cdot \nabla v - \frac{c_*}{|x|^2}uv \right)dx
$$
and
$$ \| \sqrt{H} u \|_{L^2} = \langle u, Hu \rangle ^{\frac{1}{2}}. $$
\item The notation $A \gtrsim B$ (resp. $A \lesssim B$) means $A \ge C\,B$ (resp. $A \le C\,B$) where $C$ is a positive constant depending on some initial parameters.
\item $\Re (z)$, $\Im (z)$ and $\bar z$ denote the real part, the imaginary part and the complex conjugate of $z \in \C$ respectively. 
\end{itemize}
\section{Hardy-Gagliardo-Nirenberg inequality} \label{sec:HGN}

In this section we prove Theorems \ref{thm:ground-state} and \ref{thm:HGN}. 

Following Weinstein's strategy \cite{Wei-83}, we consider  the Hardy-Gagliardo-Nirenberg interpolation problem
\begin{equation} \label{eq:CHGN-def}
C_{\rm HGN} = \inf_{u\in \cQ \setminus \{0\}} \frac{ \|\sqrt{H}u\|_{L^2}^{\theta} \|u\|_{L^2}^{1-\theta} }{{\|u\|_{L^p}}}, \quad \theta=\frac{d}{2} - \frac{d}{p} . 
\end{equation}
Recall from  \eqref{eq:Q-in-Lp}  that
$$
\|\sqrt{H} u\|_{L^2} + \|u\|_{L^2}   \gtrsim \|u\|_{L^p}.
$$
By a standard scaling argument, i.e. changing $u(x)\mapsto u(\lambda x)$ and optimizing over $\lambda>0$, we find that 
\begin{align} 
\label{eq:Hardy-improved-1}
\|\sqrt{H} u\|_{L^2}^{\theta} \|u\|_{L^2} ^{1-\theta}  \gtrsim \|u\|_{L^p}, \quad \theta=\frac{d}{2}\left( 1-\frac{2}{d}\right).
\end{align}
Thus  $C_{\rm HGN}>0$.

\subsection{Existence of ground states} \label{sec:GS} This part follows from the standard direct method in the calculus of variations, using the rearrangement inequalities and the compactness of radial functions (see \cite{Strauss-77,SuWanWil-07}, \cite[Lemma 3.1]{TraZog-15}). We will give a short proof for the reader's convenience.

Let $\{u_n\}$ be a minimizing sequence for $C_{\rm HGN}$. By the Hardy--Littlewood and the P\'olya--Szeg\"o rearrangement inequalities (see \cite[Theorem 3.4]{LieLos-01} and \cite{PolSze-51,Burchard-09}) we can assume that the functions $u_n$'s are non-negative and radially symmetric decreasing. By a scaling argument we can also assume that
$$
\|u_n\|_{L^2}=\| \sqrt{H} u\|_{L^2} =1, \quad  \|u_n\|_{L^p} \to C_{\rm HGN}^{-1}.
$$
Since $|x|\mapsto u_n(x)$ is decreasing, we have the pointwise estimate
\begin{align*}
|u_n(x)|^2 &\le \frac{1}{B(0,|x|)} \int_{|y|\le |x|} |u_n(y)|^2 dy \leq C|x|^{-d}
\end{align*}
with a constant $C>0$ independent of $n$. Therefore, by Helly's selection theorem we find that, up to a subsequence when $n
\to \infty$, we have the pointwise convergence
$$u_n(x) \to u_0(x) \quad \text{ for all }x \ne 0.$$
This implies that for any $\eps>0$, 
$$\1_{\{|x|\ge \eps\}}(u_n-u_0)\to 0 \quad \text{strongly in } L^p(\R^d)$$
as $n\to \infty$ by Dominated Convergence Theorem. On the other hand, from the uniform bound $\|u_n\|_{L^q} \le C$ for $p<q<2^*$, we find that 
$$\1_{\{|x|< \eps\}}(u_n-u_0)\to 0 \quad \text{strongly in }L^p(\R^d)$$
as $\eps\to 0$, uniformly in $n$. Thus $u_n\to u_0$  strongly in $L^p(\R^d)$, and hence 
$$\|u_0\|_{L^p}=\lim_{n\to \infty} \|u_n\|_{L^p} = C_{\rm HGN}^{-1}.$$

On the other hand, by the Banach-Alaoglu theorem, up to a subsequence as $n\to \infty$ again, we can assume that 
$$u_n\wto u_0, \quad \sqrt{H} u_n \wto \sqrt{H} u_0 \quad \text{ weakly in $L^2(\R^d)$}.$$
Hence,  we have
$$ \|u_0\|_{L^2}\le \liminf_{n\to \infty} \|u_n\|_{L^2}=1, \quad \|\sqrt{H} u_0\|_{L^2}\le \liminf_{n\to \infty} \|\sqrt{H} u_n\|_{L^2}=1.$$
In summary, we have proved that 
$$
\frac{\|\sqrt{H}u_0\|_{L^2}^{\theta} \|u_0\|_{L^2}^{1-\theta}}{\|u_0\|_{L^p}} \le  C_{\rm HGN}. 
$$

Thus $u_0$ is a minimizer for the variational problem \eqref{eq:HGN}. From the above proof, since the minimizing sequence $\{u_n\}$ are nonnegative radially symmetric decreasing, the limit $u_0$ is also nonnegative radially symmetric decreasing. 

\bigskip

\noindent{\bf Euler-Lagrange equation.}   By standard variational techniques, we can show that the above minimizer $u_0$ satisfies the Euler-Lagrange equation
\begin{equation} \label{eq:EL-non-normalized}
\theta H u_0 + (1-\theta) u_0 - (C_{\rm HGN})^{\frac{p}{2}} u_0^{p-1}=0.
\end{equation}
Here the relevant coefficients come from the constraints 
$$
\|u_0\|_{L^2} = 1= \| \sqrt{H} u_0\|_{L^2},  \quad \|u_0\|_{L^p} = C_{\rm HGN}^{-1}.
$$

If we define 
$$u_0(x)=\alpha Q (\beta x), \quad \beta=\left( \frac{1-\theta}{\theta} \right)^{\frac{1}{2}}, \quad  \alpha=(1-\theta)^{\frac{1}{p-2}}(C_{\rm HGN})^{-\frac{p}{p-2}},
$$
then \eqref{eq:EL-non-normalized} becomes
\begin{equation} \label{eq:EL-Q}
H Q + Q- Q^{p-1} =0.
\end{equation}
Moreover, $Q$ is nonnegative radially symmetric decreasing and 
\begin{equation} \label{eq:CHGN-Q}
\|Q\|_{L^2} = \beta^{\frac{d}{2}} \alpha^{-1}  \|u_0\|_{L^2} = \left( \frac{1-\theta}{\theta} \right)^{\frac{d}{4}} (1-\theta)^{-\frac{1}{p-2}}(C_{\rm HGN})^{\frac{p}{p-2}}. 
\end{equation}

\subsection{A-priori estimates} 

The following a-priori estimates will be important for the proof of the uniqueness of $Q$. 

\begin{lemma}[A-priori estimates] \label{lem:Q-apriori} Let $Q\in \cQ$ be a nonnegative radial  solution to \eqref{eq:EL-Q}. Then $Q \in C^2(\R^d \setminus \{0\})$ and $Q$ is strictly positive. Moreover, 
\begin{align} \label{eq:Q-0}
\lim_{|x| \to 0}|x|^{\frac{d-2}{2}}Q(x) \in (0,\infty).
\end{align}
and
\begin{align} \label{eq:Q-inf}
\lim_{|x| \to \infty}|x|^{m}Q(x) =0, \quad \forall m>0. 
\end{align} 
\end{lemma}

\begin{proof} Since $Q \in \cQ$, by a standard bootstrap argument, together with Sobolev embeddings, one can show that $Q \in C^2(\R^d \setminus \{0\})$.
	
The equation \eqref{eq:EL-Q} can be rewritten as 
\begin{align}
Q(x) &= (I-\Delta)^{-1} \Big( c_* |x|^{-2} Q(x) + Q(x)^{p-1}  \Big) \nn\\
&= \int_{\R^d} G(x-y)  \Big( c_* |y|^{-2} Q(y) + Q(y)^{p-1}  \Big) dy \label{eq:Q-conv}
\end{align}
where $G$ is the Green function of $I-\Delta$ (the Yukawa potential). Recall that \cite[Theorem 6.23]{LieLos-01}
$$
G(x)= \int_0^\infty (4\pi t)^{-d/2} \exp \left( - \frac{|x|^2}{4t} -t \right) dt
$$
and 
\begin{align} \label{eq:G-x-all}
0<G(x) \le \int_0^\infty (4\pi t)^{-d/2} \exp \left( - \frac{|x|^2}{4t} \right) dt = \frac{c_d}{|x|^{d-2}}, \quad \forall x\in \R^d \setminus \{ 0 \}
\end{align}
and 
\begin{align} \label{eq:G-x-large}
\lim_{|x|\to \infty} -\frac{\log G(x)}{|x|} =1.  
\end{align}

In particular, since $Q(x)$ is the convolution of $G(x)>0$ and $c_* |x|^{-2} Q(x) + Q(x)^{p-1} \ge 0$, we find that $Q(x)>0$ for all $x \in \R^d$. 

The formula \eqref{eq:Q-0} has been proved by Trachanas-Zographopoulos \cite[Theorem 1.2]{TraZog-15}. 

Now we consider the decay of $Q$ at infinity. We take $|x|$ large and decompose the integral domain in \eqref{eq:Q-conv} into $|x-y|<|x|^{1/4}$ and $|x-y|\ge |x|^{1/4}$. 

In the region $|x-y|<|x|^{1/4}$, by the triangle inequality we have 
$$|y|>|x|-|x|^{1/4}\ge |x|/2 \ge 1$$
for $|x|$ large. Combining with the upper bound \eqref{eq:G-x-all} and Newton's Theorem \cite[Theorem 9.7]{LieLos-01} (here $Q$ is radial) we can bound
 \begin{align} \label{eq:Q-decay-1}
& \int_{|x-y|<|x|^{1/4}} G(x-y)  \Big( c_* |y|^{-2} Q(y) + Q(y)^{p-1}  \Big) dy \nn\\
&\lesssim \int_{|x-y|<|x|^{1/4}} \frac{1}{|x-y|^{d-2} }\Big(  Q(y) + Q(y)^{p-1}  \Big) dy\nn \\
&\lesssim \int_{|x-y|<|x|^{1/4}} \frac{1}{|x|^{d-2} }\Big(  Q(y) + Q(y)^{p-1}  \Big) dy.
\end{align}
On the other hand, since $Q\in L^2(\R^d)\cap L^p(\R^d)$, by H\"older's inequality and keeping in mind that $p>2$, we have
 \begin{align*}
\int_{|x-y|<|x|^{1/4}} Q(y) d y \le \left( \int_{|x-y|<|x|^{1/4}} dy \right)^{1/2} \left( \int_{|x-y|<|x|^{1/4}} |Q(y)|^2 d y \right)^{1/2} \lesssim |x|^{d/8}
\end{align*}
and 
\begin{align*}
\int_{|x-y|<|x|^{1/4}} Q(y)^{p-1} d y &\le \left( \int_{|x-y|<|x|^{1/4}} dy \right)^{1/p} \left( \int_{|x-y|<|x|^{1/4}} |Q(y)|^p d y \right)^{(p-1)/p} \lesssim |x|^{d/8}.
\end{align*}
Thus from \eqref{eq:Q-decay-1} it follows that 
 \begin{align} \label{eq:Q-decay-1a}
\int_{|x-y|<|x|^{1/4}} G(x-y)  \Big( c_* |y|^{-2} Q(y) + Q(y)^{p-1}  \Big) dy \lesssim \frac{|x|^{d/8}}{|x|^{d-2} } \lesssim |x|^{-5/8}. 
\end{align}
Here we have used $d\ge 3$ in the last estimate. 

In the region $|x-y|\ge |x|^{1/4}$, using \eqref{eq:G-x-large} we have
$$
G(x-y) \le e^{-|x-y|/2} \le e^{-|x-y|/4} e^{-|x|^{1/4}/4} 
$$
for $|x|$ large.  Hence, 
 \begin{align} \label{eq:Q-decay-2}
& \int_{|x-y| \ge |x|^{1/4}} G(x-y)  \Big( c_* |y|^{-2} Q(y) + Q(y)^{p-1}  \Big) dy \nn\\
&\lesssim e^{-|x|^{1/4}/4}  \int_{\R^d} e^{-|x-y|/4}\Big( |y|^{-2} Q(y) + Q(y)^{p-1}  \Big) dy.
\end{align}
By  H\"older's inequality we can bound
$$
\int_{\R^d} e^{-|x-y|/4} Q(y)^{p-1} dy \le \left( \int_{\R^d} e^{-|x-y|p/4} dy \right)^{1/p} \left( \int_{\R^d} |Q(y)|^{p} dy \right)^{(p-1)/p} \lesssim 1
$$
and
 \begin{align*}
&\int_{\R^d} e^{-|x-y|/4} |y|^{-2} Q(y) dy \le  \int_{|y|\le 1} |y|^{-2}  Q(y) dy + \int_{|y|>1} e^{-|x-y|/4} Q(y) dy\\
&\le   \left( \int_{|y|\le 1} \frac{1}{|y|^{d-1/2}} dy \right)^{4/(2d-1)} \left( \int_{|y|\le 1} |Q(y)|^{(2d-1)/(2d-5)} dy \right)^{ (2d-5)/(2d-1)}  \\
&\quad + \left( \int_{|y|>1} e^{-|x-y|/2} dy \right)^{1/2} \left( \int_{|y|>1} |Q(y)|^{2} dy \right)^{1/2}  \lesssim  1.
\end{align*}
In the last estimate we have used that $Q\in L^q(\R^d)$ for all $2\le q<2^*$ and that
$$
\frac{2d-1}{2d-5} < \frac{2d}{d-2}=2^*
$$
for all $d\ge 3$. Thus \eqref{eq:Q-decay-2} gives us
 \begin{align} \label{eq:Q-decay-2a}
\int_{|x-y| \ge |x|^{1/4}} G(x-y)  \Big( c_* |y|^{-2} Q(y) + Q(y)^{p-1}  \Big) dy \lesssim e^{-|x|^{1/4}/4}
\end{align}
for $|x|$ large. 

Putting \eqref{eq:Q-decay-1a} and \eqref{eq:Q-decay-2a} together, we obtain
$$
Q(x)= \int_{\R^d} G(x-y)  \Big( c_* |y|^{-2} Q(y) + Q(y)^{p-1}  \Big) dy \lesssim |x|^{-5/8} + e^{-|x|^{1/4}/4} \lesssim |x|^{-5/8} 
$$
for $|x|$ large. 

Next, assume that we have proved that $Q(x) \lesssim |x|^{-m}$ for $|x|$ large, with some constant $m>0$. Then for $|x|$ large, in the region $|x-y|\le |x|^{1/4}$ using again the triangle inequality $|y|>|x|-|x|^{1/4} \ge |x|/2$ we get 
$$Q(y) + Q(y)^{p-1} \lesssim |x|^{-m}.$$
Inserting this poinwise bound into \eqref{eq:Q-decay-1} we find that 
 \begin{align} \label{eq:Q-decay-1b}
& \int_{|x-y|<|x|^{1/4}} G(x-y)  \Big( c_* |y|^{-2} Q(y) + Q(y)^{p-1}  \Big) dy \nn\\
&\lesssim \int_{|x-y|<|x|^{1/4}} \frac{1}{|x|^{d-2} }\Big( Q(y) + Q(y)^{p-1}  \Big) dy \nn\\
&\lesssim \int_{|x-y|<|x|^{1/4}} \frac{1}{|x|^{d-2} } |x|^{-m} dy \lesssim   \frac{|x|^{d/4}}{|x|^{d-2} }   |x|^{-m}  \lesssim  |x|^{-m-1/4}
\end{align}
for all $d\ge 3$. Putting the latter bound together  with \eqref{eq:Q-decay-2a} we obtain
$$
Q(x)= \int_{\R^d} G(x-y)  \Big( c_* |y|^{-2} Q(y) + Q(y)^{p-1}  \Big) dy \lesssim |x|^{-m-1/4} + e^{-|x|^{1/4}/4} \lesssim |x|^{-m-1/4} 
$$
for $|x|$ large. 

Thus in summary, we have proved that if $Q(x) \lesssim |x|^{-m}$ for $|x|$ large, then $Q(x) \lesssim |x|^{-m-1/4}$ for $|x|$ large. By induction, we conclude that for any constant $m>0$, then $Q(x) \lesssim |x|^{-m}$ for $|x|$ large. Consequently, we get \eqref{eq:Q-inf}, namely
$$
\lim_{|x| \to \infty}|x|^{m}Q(x) =0, \quad \forall m>0. 
$$ 
The proof is complete.
\end{proof}

\subsection{Uniqueness of ground state solution} \label{sec:Uni} Now we  study the uniqueness of positive radial solutions to  equation \eqref{eq:EL-Q}. Thanks to Lemma \ref{lem:Q-apriori}, it suffices to show that there exists at most one positive radial $C^2$ solution to 
\begin{equation} \label{eq:Uni1} \left\{ \begin{aligned}
- H Q + Q - Q^{p-1} &= 0 \quad \text{in } \R^d \setminus \{0\}, \\
\lim_{|x| \to 0}|x|^{\frac{d-2}{2}} Q (x) &\in (0,\infty), \\
\lim_{|x|\to \infty}|x|^m Q (x) & \in [0,\infty),
\end{aligned} \right. \end{equation}
for some $m>\frac{d+2}{2}$. Here the first equation in \eqref{eq:Uni1} is understood in the classical sense in $\R^d \setminus \{0\}$.

Note that if $Q$ is radial solution of \eqref{eq:Uni1} then using the polar coordinate $r=|x|$ we can write 
\begin{equation} \label{eq:Uni2} \left\{ \begin{aligned}
\frac{d^2}{dr^2} Q + \frac{d-1}{r} \frac{d}{dr} Q + \frac{c_*}{r^2}Q -Q + Q^{p-1}&= 0 \quad \text{in } (0,\infty), \\
\lim_{r \to 0^+}r^{\frac{d-2}{2}} Q(r) &\in (0,\infty), \\
\lim_{r \to +\infty}r^m Q(r) &\in [0,\infty).
\end{aligned} \right. \end{equation}

By putting 
\begin{equation}  \label{eq:v-Q-def}
v(r)=r^{\frac{d-2}{2}}Q(r),
\end{equation}
we deduce that $v \in C^2((0,\infty))$, $v>0$ in $(0,\infty)$ and $v$ satisfies
\begin{equation} \label{eq:Uni3} \left\{ \begin{aligned}
\frac{d^2}{dr^2} v + \frac{1}{r} \frac{d}{dr} v  - v + r^{-\frac{(d-2)(p-2)}{2}}v^{p-1}&= 0 \quad \text{in } (0,\infty), \\
v(0) &\in (0,\infty), \\
\lim_{r \to +\infty}r^{m-\frac{d-2}{2}} v(r) & \in [0,\infty). 
\end{aligned} \right. \end{equation}
We will prove that there exists at most one positive solution $v$ of  \eqref{eq:Uni3}. In the following we will use the notation $v_r$ instead of $\frac{d}{dr}v$. 

\bigskip
\noindent 
{\bf  Pohozaev identity.} By using the computation in Shioji-Watanabe \cite{ShiWat-13} with 
$$f(r)=r, \quad g(r)=-1, \quad h(r)=r^{-\frac{(d-2)(p-2)}{2}},$$
we obtain the generalized Pohozaev identity 
\begin{equation} \label{dJ}
\frac{d}{dr}J(r,u) = G(r)u(r)^2 \quad \forall r \in (0,\infty)
\end{equation}
where  
\begin{align} \label{a}
a(r) &:=  r^{\frac{d(p-2)+4}{p+2}}, \\ \label{b}
b(r) &:= \frac{d+2 - (d-2)(p-1)}{2(p+2)}r^{\frac{(d-1)(p-2)}{p+2}}, \\  \nonumber
c(r) &:=  \frac{[(d+2-(d-2)(p-1)]^2}{2(p+2)^2} r^{-\frac{d+2 - (d-2)(p-1)}{d+2}}, \\ \nonumber
G(r) &:=  b(r) + \frac{1}{2}c_r(r) - \frac{1}{2}a_r(r) \\ \nonumber
&= - \frac{(d-1)(p-2)}{p+2}r^{\frac{(d-1)(p-2)}{p+2} } -  \frac{[(d+2-(d-2)(p-1)]^3}{2(p+2)^3} r^{-\frac{d+5 - (d-3)(p-1)}{p+2}}, \\ \label{J}
J(r,v) &:= \frac{1}{2}a(r)v_r(r)^2 + b(r)v_r(r)v(r)  \\ \nonumber
&\quad + \frac{1}{2}(c(r)-a(r))v(r)^2 + \frac{1}{p} a(r)r^{-\frac{(d-2)(p-2)}{2}}v(r)^{p}. 
\end{align}

We will follow the general strategy in \cite{ShiWat-13}, but we use the following result to relax the condition (II) in \cite[Theorem 1]{ShiWat-13}.

\begin{lemma} \label{v_r(0)} Assume $2<p<2^*$ and $m>\frac{d+2}{2}$. Let $v \in C^2((0,\infty))$ be a positive solution of \eqref{eq:Uni3}. Then
\begin{equation} \label{eq:Ode0} v_r(r) + \frac{1}{r}\int_0^r s  (- v(s) + s^{-\frac{(d-2)(p-2)}{2}}v(s)^{p-1})ds = 0 \quad \forall r \in (0,\infty).
\end{equation}	
Moreover,
\begin{equation} \label{limur}  \lim_{r \to 0^+}r^\gamma  v_r(r)=0 \quad \forall \, \gamma>\frac{(d-2)(p-2)}{2}-1 
\end{equation}
and 
\begin{equation} \label{infty}
\lim_{r \to +\infty}r^{\nu}v_r(r) = 0 \quad \forall \, \nu < m - \frac{d}{2}.
\end{equation}
\end{lemma}
\begin{proof} From \eqref{eq:Uni3}, we have, for any $0<r'<r$,
\begin{equation} \label{eq:Ode1} rv_r(r) - r'v_r(r') + \int_{r'}^r s  (- v(s) + s^{-\frac{(d-2)(p-2)}{2}}v(s)^{p-1})ds = 0.
\end{equation}
Note that 
$$\int_0^r s^{1- \frac{(d-2)(p-2)}{2}}ds<\infty$$
since 
$$1-\frac{(d-2)(p-2)}{2}>-1$$
as $p<2^*$. Consequently, from \eqref{eq:Ode1}, we deduce that the limit
$$k=\lim_{r' \to 0^+}r'v_r(r')$$
exists as a real number. Letting $r' \to 0^+$ in \eqref{eq:Ode1} yields
\begin{equation} \label{eq:Ode2} v_r(r) - \frac{k}{r} + \frac{1}{r}\int_0^r s  (- v(s) + s^{-\frac{(d-2)(p-2)}{2}}v(s)^{p-1})ds = 0.
\end{equation}
Integrating this equation over $(\varepsilon, \tilde r)$ with $0<\varepsilon<\tilde r$ implies
\begin{equation} \label{eq:Ode3} v(\tilde r) - v(\varepsilon) - k\ln\left( \frac{\tilde r}{\varepsilon} \right) + \int_\varepsilon^{\tilde r} \frac{1}{r}\int_0^r s  (- v(s) + s^{-\frac{(d-2)(p-2)}{2}}v(s)^{p-1})dsdr =0.
\end{equation}
Using again $p<2^*$ we get 
$$ \left| \int_0^{\tilde r} \frac{1}{r}\int_0^r s  (- v(s) + s^{-\frac{(d-2)(p-2)}{2}}v(s)^{p-1})dsdr \right| < \infty.
$$
By letting $\varepsilon \to 0$ in \eqref{eq:Ode3} we find that $k=0$. Thus \eqref{eq:Ode0} follows from \eqref{eq:Ode2}. 

Moreover, by \eqref{eq:Ode0} and since $p<2^*$ and $\gamma>\frac{(d-2)(p-2)}{2}-1$,
$$ \lim_{r \to 0^+}r^\gamma v_r(r) = -\lim_{r \to 0^+} r^{\gamma-1}  \int_0^r s  (- v(s) + s^{-\frac{(d-2)(p-2)}{2}}v(s)^{p-1})dsdr = 0.
$$
Thus we obtain \eqref{limur}. 

Next we prove \eqref{infty}. From \eqref{eq:Uni3}, there exists $r_0>0$ large enough such that 
\begin{align} \label{v-est}
v(r) \lesssim r^{\frac{d-2}{2}-m} \quad \forall r \in [r_0,\infty).
\end{align}

We use \eqref{eq:Ode1} for large numbers $r_0< r'<r$. Since $m>\frac{d+2}{2}$, we deduce from \eqref{v-est} and the third equality in \eqref{eq:Uni3} that 
\begin{align} \label{eq:sv-1}
\lim_{r \to +\infty} \int_{r'}^r s v(s) ds \lesssim \lim_{r \to +\infty} \int_{r'}^r s^{\frac{d-2}{2}-m +1}ds = \frac{2(r')^{\frac{d+2-2m}{2}}}{2m - d - 2}
\end{align}
and
\begin{align} \label{eq:sv-2}
\lim_{r \to +\infty} \int_{r'}^r s^{1-\frac{(d-2)(p-2)}{2}}v(s)^{p-1}ds \lesssim \lim_{r \to +\infty} \int_{r'}^r s^{\frac{d-2m(p-1)}{2}}ds=\frac{2(r')^{\frac{d+2-2m(p-1)}{2}}}{2m(p-1)-d-2}. 
\end{align}
Therefore, we see from \eqref{eq:Ode1}, using \eqref{eq:sv-1} and \eqref{eq:sv-2}, that the limit
$$ K=\lim_{r \to +\infty}rv_r(r)
$$
exists as a real number. Letting $r \to +\infty$ in \eqref{eq:Ode1} implies
\begin{align} \label{eq:K-1}
\frac{K}{r'} - v_r(r') + \frac{1}{r'}\int_{r'}^\infty s  (- v(s) + s^{-\frac{(d-2)(p-2)}{2}}v(s)^{p-1})ds = 0.
\end{align}
By integrating over $(r,R)$ with $r_0<r<R$, we find
\begin{align} \label{eq:K-2}
K\ln\left( \frac{R}{r} \right) - v(R) + v(r) + \int_r^R \frac{1}{r'}\int_{r'}^\infty s  (- v(s) + s^{-\frac{(d-2)(p-2)}{2}}v(s)^{p-1})ds dr'= 0.
\end{align}
Since $m>\frac{d+2}{2}$, we derive
$$ \left| \int_r^\infty \frac{1}{r'}\int_{r'}^\infty s  (- v(s) + s^{-\frac{(d-2)(p-2)}{2}}v(s)^{p-1})ds dr' \right| < \infty.
$$
This, together with the fact that $v$ decays at infinity and  \eqref{eq:K-2}, implies $K=0$. Consequently, we infer from \eqref{eq:K-1} that
\begin{align} \label{eq:mu0}
(r')^\nu v_r(r') = (r')^{\nu-1}\int_{r'}^\infty s  (- v(s) + s^{-\frac{(d-2)(p-2)}{2}}v(s)^{p-1})ds. 
\end{align}
From \eqref{eq:sv-1} and \eqref{eq:sv-2} and the fact that $p>2$, we deduce
\begin{align} 
(r')^{\nu-1} \left| \int_{r'}^\infty s  (- v(s) + s^{-\frac{(d-2)(p-2)}{2}}v(s)^{p-1})ds \right| \lesssim (r')^{\nu-m + \frac{d}{2}}.
\end{align}
Consequently, as $\nu < m - \frac{d}{2}$, it follows
\begin{align} \label{eq:mu1}
\lim_{r' \to +\infty} (r')^{\nu-1} \left| \int_{r'}^\infty s  (- v(s) + s^{-\frac{(d-2)(p-2)}{2}}v(s)^{p-1})ds \right| =0.
\end{align} 
Combining \eqref{eq:mu0} and \eqref{eq:mu1} leads to \eqref{infty}.
\end{proof}
Now we are ready to conclude

\begin{proposition}[Uniqueness of \eqref{eq:Uni3}] \label{lem:unique-v} Assume $2<p<2^*$ and $m>\frac{d+2}{2}$. Then problem \eqref{eq:Uni3} admits at most one positive solution.
\end{proposition}

\begin{proof}

Let $v$ and $\tilde v$ be two positive solutions of \eqref{eq:Uni3}. We will prove that $v=\tilde v$ by using a shooting argument. 

\bigskip

\noindent
{\bf Step 1.} We show that if $v(0)=\tilde v(0)$ then $v=\tilde v$ in $(0,\infty)$.	

\begin{proof} Let $R>0$. From Lemma \ref{v_r(0)}, we see that, for any $0<r<R$, 
\begin{align*}
v(r) &= v(0) - \int_0^{r} \frac{1}{\sigma}\int_0^\sigma s  (- v(s) + s^{-\frac{(d-2)(p-2)}{2}}v(s)^{p-1})dsd\sigma \\
&= v(0) - \int_0^r \left( \int_s^r \frac{1}{\sigma}d\sigma \right) s  (- v(s) + s^{-\frac{(d-2)(p-2)}{2}}v(s)^{p-1})ds \\
&= v(0) - \int_0^r \left(\ln\frac{r}{s} \right) s  (- v(s) + s^{-\frac{(d-2)(p-2)}{2}}v(s)^{p-1})ds.
\end{align*}
Similarly, we have
$$
\tilde v(r) =  \tilde v(0) - \int_0^r \left(\ln\frac{r}{s} \right) s  (- \tilde v(s) + s^{-\frac{(d-2)(p-2)}{2}} \tilde v(s)^{p-1})ds.
$$
Keeping in mind that $v(0) = \tilde v(0)$, $v,\tilde v$ are bounded in $[0,R]$ and $p>2$, we deduce from the above equalities that
\begin{align*} |v(r) -\tilde v(r)| &\leq \int_0^r \left(\ln\frac{r}{s} \right) s  (|v(s) - \tilde v(s)| + s^{-\frac{(d-2)(p-2)}{2}}|v(s)^{p-1} - \tilde v(s)^{p-1}|)ds \\
&\leq C(R)\int_0^r \left(\ln\frac{r}{s} \right) s(1+s^{-\frac{(d-2)(p-2)}{2}})|v(s) - \tilde v(s)|ds.
\end{align*}
Since $p<2^*$ we have  
$$
1-\frac{(d-2)(p-2)}{2}>-1,
$$
and hence
$$ \int_0^r \left(\ln\frac{r}{s} \right) s(1+s^{-\frac{(d-2)(p-2)}{2}})ds <\infty.
$$
Therefore, by Gronwall's inequality, we find that $v=\tilde v$ in $[0,R)$. Since $R>0$ is arbitrary, we deduce that $v=\tilde v$ in $[0,\infty)$. 
\end{proof}

\noindent
{\bf Step 2.} We show that
\begin{equation} \label{v/v} \frac{d}{dr} \left( \frac{\tilde v(r)}{v(r)}  \right) = \frac{1}{rv(r)^2 } \int_0^r  s^{1-\frac{(d-2)(p-2)}{2}} (v(s)^{p-1} - \tilde v(s)^{p-1}) v(s) \tilde v(s)  ds \quad \forall r \in (0,\infty).
\end{equation}

\begin{proof}
Since $v$ and $\tilde v$ are two solutions of \eqref{eq:Uni3}, we obtain
\begin{align} \label{v1}
(rv_r)_r + r(-v +  r^{-\frac{(d-2)(p-2)}{2}}v^{p-1}) = 0, \\ \label{tildev1}
(r \tilde v_r)_r + r(-\tilde v +  r^{-\frac{(d-2)(p-2)}{2}} \tilde v^{p-1}) = 0.
\end{align}
By multiplying \eqref{v1}  by $\tilde v$ and multiplying  \eqref{tildev1} by $v$, then integrating over $[r',r]$ with $0<r'<r$, we obtain
\begin{equation} \label{vtv} \begin{aligned} 
&r(v_r(r)\tilde v(r) - v(r) \tilde v_r(r)) - r'(v_r(r')\tilde v(r') - v(r') \tilde v_r(r')) \\
&+ \int_{r'}^r s^{1-\frac{(d-2)(p-2)}{2}} (v(s)^{p-1} - \tilde v(s)^{p-1}) v(s) \tilde v(s)  ds =0.
\end{aligned} \end{equation}
Since $p<2^*$, thanks to Lemma \ref{v_r(0)},  $\lim_{r' \to 0^+}(r'v_r(r')) = \lim_{r' \to 0^+}(r' \tilde v_r(r'))=0$. Therefore, by letting $r' \to 0$ in \eqref{vtv}, we obtain \eqref{v/v}. 
\end{proof}

\noindent
{\bf Step 3.} We show that if $\tilde v(0)<v(0)$ then  
\begin{equation} \label{v/v>0}
\frac{d}{dr} \left( \frac{\tilde v(r)}{v(r)}  \right) > 0 \quad \forall r \in (0,\infty).
\end{equation}	

\begin{proof} If $v$ does not intersect $\tilde v$ then $v(r)> \tilde v(r)$ for any $r \in (0,\infty)$ by the continuity. In this case \eqref{v/v>0} follows immediately from \eqref{v/v}. 

Now we suppose that $v$ intersects $\tilde v$ at least one point. Denote by $r_1 \in (0,\infty)$ the first intersection of $v$ and $\tilde v$. Put $w= \tilde v/ v$. By \eqref{v/v}, $w_r(r)>0$ for any $r \in (0,r_1]$. 

Suppose by contradiction that \eqref{v/v>0} does not hold. Then there exists $r_2 >r_1$ such that $w_r(r)>0$ for any $r \in (0,r_2)$ and $w_r(r_2)=0$. Hence we see that $w(r_2)v(r) > \tilde v(r)$ for any $r \in (0,r_2)$ and
\begin{equation} \label{r_2} w(r_2)>1,\quad w(r_2)v(r_2)=\tilde v(r_2), \quad w(r_2)v_r(r_2)=\tilde v_r(r_2). 
\end{equation}

From the generalized Pohozaev identity  \eqref{dJ}, for $r<r_2$ we have
\begin{equation} \label{eq:wJ} \begin{aligned} &w(r_2)^2J(r_2,v) - J(r_2,\tilde v) \\
&=\int_{r}^{r_2} G(s)(w(r_2)^2 v(s)^2 - \tilde v(s)^2 )ds + w(r_2)^2J(r,v) - J(r,\tilde v).
\end{aligned} \end{equation}
The left hand side of \eqref{eq:wJ} can be estimated by \eqref{r_2} as
\begin{equation} \label{LHS} 
w(r_2)^2J(r_2,v) - J(r_2,\tilde v) = \frac{w(r_2)^2 - w(r_2)^{p}}{p}a(r_2)r_2^{-\frac{(d-2)(p-2)}{2}}v(r_2)^{p}<0.
\end{equation} 
For the right hand side of \eqref{eq:wJ}, since $w_r(r)>0$ for any $r \in (0,r_2)$, it follows that $0<w(r)<w(r_2)$ for any $0<r<r_2$ and $w(r)v(s)<\tilde v(s)$ for any $s \in (r,r_2)$. From \eqref{dJ} and $G(r)<0<v(r)$ it follows that 
$$\frac{d}{dr}J(r,v) = G(r) v(r)^2 <0, \quad \forall r \in (0,\infty).$$
Consequently, $J(r,v)$ is strictly decreasing with respect to $r$. Moreover, since  $v$ satisfies the third limit in \eqref{eq:Uni3} with $m>\frac{d+2}{2}$ and the limit \eqref{infty} with $\nu<m-\frac{d}{2}$, we have
\begin{equation}  \label{lim-infinity} \begin{aligned}
&\lim_{r \to +\infty}a(r)^{\frac{1}{2}}v_r(r)=\lim_{r \to +\infty} r^{\frac{d(p-2)+4}{2(p+2)}}v_r(r) = 0, \\
&\lim_{r \to +\infty}b(r)v_r(r)v(r) \lesssim \lim_{r \to +\infty}r^{\frac{(d-1)(p-2)}{p+2}+\frac{d-2}{2}-m}v_r(r)=0,\\
&\lim_{r \to +\infty} c(r)v(r)^2 \lesssim \lim_{r \to +\infty}r^{-\frac{d+2 - (d-2)(p-1)}{d+2} + d-2 - 2m} = 0, \\
&\lim_{r \to +\infty} a(r)v(r)^2 \lesssim \lim_{r \to +\infty}r^{\frac{d(p-2)+4}{p+2} + d-2 - 2m} = 0,\\
&\lim_{r \to +\infty}a(r)r^{-\frac{(d-2)(p-2)}{2}}v(r)^p \lesssim \lim_{r \to +\infty} r^{\frac{d(p-2)+4}{p+2} - \frac{(d-2)(p-2)}{2} + \frac{(d-2-2m)p}{2}}= 0.
\end{aligned} \end{equation}
Inserting these limits into the formula of $J(r,v)$ we obtain
$$\lim_{r \to +\infty}J(r,v) = 0.$$
Similarly
$$ \lim_{r \to +\infty}J(r,\tilde v) = 0.
$$
Therefore 
\begin{equation} \label{J>0} J(r,v) > 0 \quad \forall r \in (0,\infty).
\end{equation}
This leads to
$$ \int_{r}^{r_2} G(s)v(s)^2ds + J(r,v) = J(r_2,v) > 0. 
$$
Combing the above estimates, we can estimate the right hand side of \eqref{eq:wJ} as follows
\begin{equation} \label{RHS1} \begin{aligned}
&\int_{r}^{r_2} G(s)(w(r_2)^2 v(s)^2 - \tilde v(s)^2 )ds + w(r_2)^2J(r,v) - J(r,\tilde v) \\
&=w(r_2)^2 \left(\int_{r}^{r_2} G(s)v(s)^2ds + J(r,v) \right) - \int_{r}^{r_2}G(s)\tilde v(s)^2 ds - J(r,\tilde v) \\
&\geq w(r)^2\left(\int_{r}^{r_2} G(s)v(s)^2ds + J(r,v) \right) - \int_{r}^{r_2}G(s)\tilde v(s)^2 ds - J(r,\tilde v) \\
&=\int_{r}^{r_2} G(s)(w(r)^2 v(s)^2 - \tilde v(s)^2 )ds + w(r)^2J(r,v) - J(r,\tilde v) \\
&\geq w(r)^2J(r,v) - J(r,\tilde v)
\end{aligned} \end{equation}
for all $r \in (0,r_2)$. Here in the last estimate, we have used $G(s)<0$ and $w(s)v(s)<\tilde v(s)$ for all $s \in (r,r_2)$.

Next, by \eqref{J}, we have
\begin{equation} \label{RHS2} \begin{aligned}
w(r)^2J(r,v) - J(r,\tilde v)  &= \frac{1}{2}a(r)(w(r)^2 v_r(r)^2 - \tilde v_r(r)^2) \\
&\quad + b(r)( w(r)^2 v_r(r)v(r) - \tilde v_r(r)\tilde v(r))  \\ 
&\quad + \frac{1}{p} a(r)r^{-\frac{(d-2)(p-2)}{2}}(w(r)^2 v(r)^{p} - \tilde v(r)^{p}).
\end{aligned} \end{equation}

Since $p<2^*$ it follows that 
$$\frac{d(p-2)+4}{2(p+2)} > \frac{(d-2)(p-2)}{2}-1,$$
and hence by \eqref{a} and \eqref{limur} we deduce
$$
\lim_{r \to 0^+} a(r)^{\frac{1}{2}}v_r(r) = \lim_{r \to 0^+}r^{\frac{d(p-2)+4}{2(p+2)}}v_r(r) = 0. 
$$
Similarly 
$$
\lim_{r \to 0^+} a(r)^{\frac{1}{2}}\tilde v_r(r) = \lim_{r \to 0^+}r^{\frac{d(p-2)+4}{2(p+2)}}\tilde v_r(r) = 0. 
$$
Moreover, $p<2^*$ also implies that 
$$\frac{(d-1)(p-2)}{p+2}>\frac{(d-2)(p-2)}{2}-1,$$
and hence by \eqref{b} and \eqref{limur} we deduce
$$
\lim_{r \to 0^+} b(r)v_r(r) = \lim_{r \to 0^+}\frac{d+2 - (d-2)(p-1)}{2(p+2)}r^{\frac{(d-1)(p-2)}{p+2}}v_r(r) = 0. 
$$
Similarly 
$$
\lim_{r \to 0^+} b(r)\tilde v_r(r) = \lim_{r \to 0^+}\frac{d+2 - (d-2)(p-1)}{2(p+2)}r^{\frac{(d-1)(p-2)}{p+2}}\tilde v_r(r) = 0. 
$$
Since $p<2^*$, we have 
$$\frac{d(p-2)+4}{p+2} - \frac{(d-2)(p-2)}{2}>0,$$ 
which implies
$$ \lim_{r \to 0^+} a(r)r^{-\frac{(d-2)(p-2)}{2}} = \lim_{r \to 0^+} r^{\frac{d(p-2)+4}{p+2} - \frac{(d-2)(p-2)}{2}} = 0.
$$
Combining the above equalities and taking into account that $v,\tilde v, w$ are bounded near $0$, by letting $r \to 0^+$ in \eqref{RHS2} we deduce
\begin{equation} \label{RHS3}
\lim_{r \to 0^+} (w(r)^2J(r,v) - J(r,\tilde v)) = 0.
\end{equation}
From \eqref{eq:wJ}, \eqref{RHS1} and \eqref{RHS3}, we deduce that
$$ w(r_2)^2J(r_2,v) - J(r_2,\tilde v) \geq 0,
$$
which contradicts \eqref{LHS}. Therefore, we conclude \eqref{v/v>0}. 
\end{proof}

\noindent
{\bf Step 4: Conclusion.}  Suppose by contradiction that there exist two distinct solutions $v$ and $\tilde v$ of \eqref{eq:Uni3}. By Step 1, we know that $v(0) \neq \tilde v(0)$. Without loss of generality, we may assume that $0<\tilde v(0) < v(0)$. We see from \eqref{J>0} that $J(r,\tilde v) >0$ for any $r \in (0,\infty)$. 	
Define
\begin{equation} \label{X} X(r):=J(r,\tilde v) \left( \frac{v(r)}{\tilde v(r)} \right)^2 - J(r,v), \quad r \in (0,\infty).
\end{equation}
By inserting \eqref{J} into \eqref{X}, we deduce
\begin{align*} X(r) &= \frac{1}{2}a(r)  \left( \frac{v(r)^2 }{\tilde v(r)^2} \tilde v_r(r)^2   - v_r(r)^2 \right) + b(r) v(r) \left( \frac{v(r) }{\tilde v(r)} \tilde v_r(r)   - v_r(r) \right) \\
&\quad +\frac{1}{p}a(r)r^{-\frac{(d-2)(p-2)}{2}} v(r)^2 (\tilde v(r)^{p-2} - v(r)^{p-2}).
\end{align*}
By using the argument as in Step 3, we obtain
\begin{equation} \label{X0} \lim_{r \to 0^+}X(r) =  0.
\end{equation}
From \eqref{v/v>0}, we find 
$$ \frac{v(r)}{\tilde v(r)} < \frac{v(r_0)}{\tilde v(r_0)} \quad \forall r \in (r_0,\infty).
$$
By combining this and \eqref{lim-infinity}, we deduce
\begin{equation} \label{Xinf}
\lim_{r \to +\infty}X(r) = 0. 
\end{equation}
On the other hand, we observe that
$$ \frac{d}{dr}X(r) = 2J(r,\tilde v) \frac{v(r)}{\tilde v(r)} \frac{d}{dr}  \left( \frac{v(r)}{\tilde v(r)} \right).
$$	
This, together with the conclusion in Step 3, yields $\frac{d}{dr}X(r)<0$ for any $r \in (0,\infty)$. However, the latter fact contradicts \eqref{X0} and \eqref{Xinf}. Thus we conclude that problem \eqref{eq:Uni3} admits at most one solution. This ends the proof of Proposition \ref{lem:unique-v}.
\end{proof}

Now let us conclude. 

\begin{proof}[Proof of Theorem \ref{thm:ground-state}] The existence has been proved in subsection \ref{sec:GS}, in particular \eqref{eq:EL-Q}. The uniqueness in the critical case $c=c_*$ has been derived  from the transformation \eqref{eq:v-Q-def}, Lemma \ref{lem:Q-apriori} and Proposition \ref{lem:unique-v}.
\end{proof}

\begin{remark}	
In the case $0<c<c_*$ we can replace \eqref{eq:v-Q-def} by \begin{equation}  \label{eq:v-Q-def}
v(r)=r^{\kappa}Q(r), 
\end{equation}
with $\kappa$ is defined in \eqref{kappa}, and then obtain the uniqueness by the same way.  
\end{remark}

\begin{proof}[Proof of Theorem \ref{thm:HGN}] The existence of a minimizer for  \eqref{eq:HGN}  has been proved in subsection \ref{sec:GS}. The sharp constant follows from \eqref{eq:CHGN-Q}. It remains to prove that any minimizer for  \eqref{eq:HGN}  has the form $u(x)=zQ(\lambda x)$. This part follows a standard technique,  but let us quickly explain for the reader's convenience. If $u$ is a minimizer for  \eqref{eq:HGN}, then by the convexity of gradients \cite[Theorem 7.8]{LieLos-01} we know that $|u|$ is also a minimizer. Moreover, by the rearrangement inequality \cite[Theorem 3.4]{LieLos-01} and the fact that $|x|^{-2}$ is strictly radially symmetric decreasing, we know that $|u|$ must be radially symmetric decreasing. Up to dilations, $|u|$ is nonnegative radial solution to equation \eqref{eq:GS}. Thus the uniqueness result in Theorem \ref{thm:ground-state} ensures that $|u|$ equals $Q$ up to dilations. In particular, we know that $|u|>0$. 

Note that if $|u|>0$ and 
$$
\|\nabla u\|_{L^2}= \| \nabla |u|\|_{L^2} 
$$ 
then $u/|u|$ is a constant. This can be seen, e.g. as in  \cite[Proposition 3]{CsoGen-18}, using the identity 
$$
|\nabla u|^2 = |\nabla (|u|w)|^2 = |\nabla (|u|) w + |u| (\nabla w)|^2 = |\nabla |u||^2 + |v|^2 |\nabla w|^2.
$$
with $w=u/|u|$. Here the cross term disappears since 
$$2 \Re (\overline w) \nabla w =\nabla |w|^2 = 0$$
as $|w|^2=1$. Thus from $\|\nabla u\|_{L^2}= \| \nabla |u|\|_{L^2}$ we find that $w$ is a constant, namely $u/|u|$ is a constant. Thus $u(x)=zQ(\lambda x)$. 
\end{proof}

\section{Basic properties of the NLS} \label{sec:basic}

In this section we prove Theorem \ref{thm:NLS-basic}.

\subsection{Local well-posedness} 

As previously explained, due to the sharpness of the Hardy potential $H$, the energy space associated to \eqref{eq:NLS} is $\cQ$ which is strictly larger than $H^1(\R^d)$. Therefore, the well-posedness results in \cite{Cazenave-03} does not apply directly to our case. 

Nevertheless, the existence and uniqueness of a local weak solution 
$$u \in C([0,T);\cQ) \cap C^1([0,T);\cQ^*)$$ of \eqref{eq:NLS} 
in a small time interval $(0,T)$ can be obtained by adapting the fixed point argument in \cite{Cazenave-03}. We refer to  \cite[Theorems 2.2, 2.3]{OkaSuzYok-12} (see also \cite[Section 4.3]{Suz-16}) for details\footnote{Conditions (G1)--(G5) in \cite[Theorem 2.2]{OkaSuzYok-12} are verified under the assumption $p<2^*$.}. Moreover, the conservation laws  \eqref{eq:conservation} also follow from  \cite[Theorem 2.3]{OkaSuzYok-12}.

Fortunately, the local existence can be derived by examining abstract assumptions stated in \cite{OkaSuzYok-12}. The uniqueness is strongly based on the Strichartz estimates for $H$ which were recently established in \cite{Suz-16,Miz-17}. The blowup alternative follows from a standard argument.

Next, we show that the short-time solution obtained previously can be extended uniquely to a maximum life time $T^*$. This step is nontrivial as the fixed point argument only works in short-time. Normally this requires a further argument using Strichartz estimates, as explained carefully for the usual NLS in \cite{Cazenave-03}.  Note that the Strichartz estimates with inverse square potentials are more subtle than that of the usual NLS, and indeed there is no end-point estimates as proved by Burq, Planchon, Stalker, and Tahvildar-Zadeh \cite{BurPlaStaTah-03}. Fortunately, the following non-end-point estimates are sufficient for our purpose.  

As usual, let $d$ denote the dimension of the space $\R^d$, we call a pair $(q,r)$  {\em admissible} if
\begin{equation} \label{admissible}
q,r \geq 2, \quad \frac{2}{q} + \frac{d}{r} = \frac{d}{2}, \quad (q,r,d) \neq (2,\infty,2).
\end{equation}
We recall the following results of Suzuki \cite[Proposition 4.8]{Suz-16}  (see also  \cite[Corollary 2.3]{Miz-17}). 

\begin{lemma}[Strichartz estimates] Assume $(q,r)$ and $(\tilde q, \tilde r)$ are admissible pairs and $q,\tilde q>2$. Let $I \subset \R$ be a time interval containing $0$. Then the following estimates hold
\begin{align} \label{est:Str1}
\|  e^{-itH} \psi \|_{L^q(I;L^r(\R^d))} &\leq C \| \psi \|_{L^2(\R^d)}, \\ \label{est:Str2}
\left \|  \int_0^t e^{-i(t-s)H}\Psi(s)   \right \|_{L^q(I;L^r(\R^d))} &\leq C\| \Psi \|_{L^{\tilde q'}(I;L^{\tilde r'}(\R^d))},
\end{align} 
for all $\psi \in L^2(\R^d)$ and $\Psi \in L^{\tilde q'}(I;L^{\tilde r'}(\R^d))$.
\end{lemma}

Using the above Strichartz estimates we obtain the following technical result. 
\begin{lemma} \label{lem:tec}Assume that $u,v\in C([0,T);\cQ) \cap C^1([0,T);\cQ^*)$ be two weak solutions of \eqref{eq:NLS} in $(0,T)$, possible with different initial data $u(0)$ and $v(0)$,  such that
$$u(\tau)=v(\tau),\quad \text{ for some } \tau \in [0,T).$$ 
Then there exists $\theta \in (0,T-\tau)$  such that 
$$u(t)=v(t), \quad \text { for all } t \in [\tau,\tau+\theta].$$ 
\end{lemma}

\begin{proof} By Duhamel's formula we can write, for $0 \le t \le T-\tau$,
\begin{align*} u(t+\tau)&=e^{-i(\tau+t)H}u_0 + i \int_0^{\tau+t}e^{-i(\tau+t-s)H}|u(s)|^{p-1}u(s)ds \\
&=e^{-i t H} \left(e^{-i\tau H}u_0 + i \int_0^{\tau}e^{-i(\tau-s)H}|u(s)|^{p-1}u(s)ds\right) \\
&\;\;\; +  i \int_{0}^{t}e^{-i(t-s)H}|u(s+\tau)|^{p-1}u(s+\tau)ds \\
&= e^{-i t H}u(\tau) +  i \int_{0}^{t}e^{-i(t-s)H}|u(s+\tau)|^{p-1}u(s+\tau)ds.
\end{align*}
Similarly, we have
$$  v(t+\tau)=e^{-i t H}v(\tau) +  i \int_{0}^{t}e^{-i(t-s)H}|v(s+\tau)|^{p-1}v(s+\tau)ds.
$$
Let $\theta \in (0,T-\tau]$ and put $\tilde u(t)=u(t+\tau)$ and $\tilde v(t)=v(t+\tau)$. For $t \in [0, \theta]$,  
\begin{align*} |\tilde u(t) - \tilde v(t)| = \left| \int_0^t e^{-i(t-s)H}(|\tilde u(s)|^{p-1}\tilde u(s) - |\tilde v(s)|^{p-1}\tilde v(s))ds   \right|.
\end{align*}
Therefore, by using \eqref{est:Str2}, the elementary inequality 
$$| |a|^{p-1}a - |b|^{p-1}b| \le p(|a|^{p-1}+|b|^{p-1})|a-b| \quad \forall a,b \in \R$$
and H\"older inequality, we obtain 
\begin{align*}
&\| \tilde u(t) - \tilde v(t) \|_{L^q((0,\theta);L^r(\R^d))} \\
&\le C \| |\tilde u|^{p-1}\tilde u - |\tilde v|^{p-1}\tilde v  \|_{L^{q'}((0,\theta);L^{r'}(\R^d))} \\
&\le C \theta^{\frac{q-q'}{qq'}} \left( \| \tilde u \|_{L^\infty((0,\theta);L^r(\R^d))}^{p-1} + \| \tilde v \|_{L^\infty((0,\theta);L^r(\R^d))}^{p-1}  \right) \| \tilde u - \tilde v\|_{L^q((0,\theta);L^r(\R^d))} \\
&\le C \theta^{\frac{q-q'}{qq'}} \left( \| u \|_{L^\infty((0,T);\cQ)}^{p-1} + \| v \|_{L^\infty((0,T);\cQ)}^{p-1}  \right) \| \tilde u - \tilde v\|_{L^q((0,\theta);L^r(\R^d))}.
\end{align*}
Let $\theta_0>0$ be such that
$$ C \theta_0^{\frac{q-q'}{qq'}} \left( \| u \|_{L^\infty((0,T);\cQ)}^{p-1} + \| v \|_{L^\infty((0,T);\cQ)}^{p-1}  \right) = \frac{1}{2}.
$$
Here we note that $\theta_0$ does not depend on $\tau$.

Hence, for any $\theta \leq \min\{ \theta_0, T-\tau \}$, it follows that
$$ \| \tilde u(t) - \tilde v(t) \|_{L^q((0,\theta);L^r(\R^d))} \leq \frac{1}{2}\| \tilde u(t) - \tilde v(t) \|_{L^q((0,\theta);L^r(\R^d))},
$$
which in turn implies $\tilde u=\tilde v$ on $[0,\theta]$. Therefore $u=v$ in $[\tau,\tau+\theta]$. This completes the proof of the technical lemma. 
\end{proof}

Now we can conclude the uniqueness. 

\begin{lemma}[Uniqueness] For any given initial datum $u_0\in \cQ$ and for any $T>0$, the equation \eqref{eq:NLS} has at most one weak solution $u\in C([0,T);\cQ) \cap C^1([0,T);\cQ^*)$ in $(0,T)$. 
\end{lemma}

\begin{proof} Assume that $u,v\in C([0,T);\cQ) \cap C^1([0,T);\cQ^*)$ are two solutions with the same initial datum $u_0$. Set
$$ \tau^*:=\sup\{ \tau \in (0,T): u = v \quad \text{in } (0,\tau)  \}.
$$
We know that $0<\tau^* \leq T$. We suppose by contradiction that $\tau^* < T$. Then there exist $\tilde \tau$ and $\varepsilon$ such that  $0<\varepsilon<\theta_0$ and $\tilde \tau < \tau^* < \min\{ \tilde \tau + \varepsilon, T - \varepsilon\}$ and $u=v$ in $(0,\tilde \tau]$. By Lemma \ref{lem:tec}, one can choose $\theta = \min\{ \theta_0, T- \tau^* \}$ such that $u=v$ in $[\tilde \tau, \tilde \tau+ \theta]$. Therefore $u=v$ in $(0,\tilde \tau + \theta]$. However, 
$$ \tilde \tau + \theta = \tilde \tau + \min\{ \theta_0, T- \tau^* \} > \tau^*,
$$
which leads to a contradiction. Thus $\tau^*=T$ and consequently $u=v$ in $[0,T)$. \end{proof}

\bigskip
\noindent
{\bf Unique continuation.} Next, for every given initial datum $u_0\in \cQ$, we can define 
$$
 T^*=T^*(u_0):=\sup\{ \, T>0: \textrm{there exists a local weak solution of } \eqref{eq:NLS}\text{  in } (0,T)  \}.
$$
From the above analysis we obtain the uniqueness of the weak solution in $(0,T^*)$. This, combined with \cite[Theorem 2.3]{OkaSuzYok-12}, implies that $u \in C([0,T^*);\cQ) \cap C^1([0,T^*);\cQ^*)$. Moreover,  the conservation laws in \eqref{eq:conservation} hold for every $t \in (0,T^*)$. Thus in summary, for given $u_0\in \cQ$, there exists a unique weak solution $u$ of \eqref{eq:NLS} in the maximal time interval $[0,T^*)$.  

\bigskip
\noindent
{\bf Blow-up alternative.} Now for any $u_0 \in \cQ$, let $u$ be a weak solution of \eqref{eq:NLS} in the maximal time interval $[0,T^*)$. We prove that if $T^*<\infty$, then 
\begin{equation} \label{eq:blowup-proof} 
\lim_{t \nearrow T^*}\| \sqrt{H} u \|_{L^2}=\infty.
\end{equation}

We observe from \cite[Theorem 2.2]{OkaSuzYok-12} that for any $\tau \in (0,T^*)$ and $\varphi \in \cQ$ with $\| \varphi \|_{\cQ} \leq M$ for some $M>0$, there exists $T_M>0$ independent of $\tau$ such that problem
\begin{equation} \label{eq:tau} \left\{  \begin{aligned}
i \partial_t u(t,x) &= Hu(t,x) - |u(t,x)|^{p-1}u(t,x) \quad && x \in \R^d,\; t >\tau, \\
u(\tau,x) &=\varphi(x) && x \in \R^d,
\end{aligned} \right. \end{equation}
admits a local weak solution in $(\tau,\tau+T_M)$.
 
If $T^*<\infty$, we suppose by contradiction that there exist $M>0$ and an increasing sequence $\{ \tau_k\}$ converging to $T^*$ such that $\| u(\tau_k) \|_{\cQ} \leq M$ for all $k \geq 1$. We can choose $k$ large enough such that $\tau_k + T_M>T^*$. By the above observation,  problem \eqref{eq:tau} with $\tau=\tau_k$ admits a solution in $(\tau_k,\tau_k+T_M)$. Consequently, problem \eqref{eq:NLS} has a weak solution in $(0,\tau_k + T_M)$, which contradicts the maximality of $T^*$. Thus \eqref{eq:blowup-proof} holds true.

\subsection{Global existence} Now we come to part (ii) of Theorem \ref{thm:NLS-basic}. By the blowup alternative, it is sufficient to show that $\| \sqrt{H} u(t) \|_{L^2}$ remains bounded uniformly in $t$. Our starting point is the following estimate 
\begin{align} \label{eq:E-below-by-HGN}
E(u_0)=E(u(t)) &=\frac{1}{2} \| \sqrt{H}u(t) \|_{L^2}^2 - \frac{1}{p}  \|u(t)\|_{L^{p}}^p \nn\\
&\ge \frac{1}{2} \| \sqrt{H}u(t) \|_{L^2}^2 - \frac{1}{pC_{\rm HGN}^p}  \| \sqrt{H}u(t) \|_{L^2}^{p\theta} \|u(t) \|_{L^2}^{p(1-\theta)} \nn\\
&= \frac{1}{2} \| \sqrt{H}u(t) \|_{L^2}^2 - \frac{1}{pC_{\rm HGN}^p}  \| \sqrt{H}u(t) \|_{L^2}^{p\theta} \|u_0 \|_{L^2}^{p(1-\theta)}
\end{align} 
which follows from the conservation laws \eqref{eq:conservation} and the Hardy-Gagliardo-Nirenberg inequality \eqref{eq:HGN}. Here recall that $\theta=d/2-d/p$.

\bigskip
\noindent
{\bf Case 1:} $2<p<2+4/d$. In this case $p\theta = pd/2-d<2$, and hence for any $\eps>0$ small we have 
$$
\| \sqrt{H}u(t) \|_{L^2}^{pd/2-d} \le \eps \| \sqrt{H}u(t) \|_{L^2}^2 + C_\eps.
$$
Inserting this in \eqref{eq:E-below-by-HGN} implies that $\| \sqrt{H}u(t) \|_{L^2}$ is bounded uniformly in $t$.  

\bigskip

\noindent
{\bf Case 2:} $p=2+4/d$. In this case, $p\theta=2$ and the sharp constant in \eqref{eq:CHGN} satisfies 
$$
pC_{\rm HGN}^p = 2 \|Q\|_{L^2}^{\frac{4}{d}}.
$$
Therefore, the lower bound \eqref{eq:E-below-by-HGN} boils down to 
$$
E(u_0)=E(u(t)) \ge \frac{1}{2} \| \sqrt{H}u(t) \|_{L^2}^2 - \frac{1}{2}   \| \sqrt{H}u(t) \|_{L^2}^{2}  \left( \frac{\|u_0\|_{L^2}}{\|Q\|_{L^2}} \right)^{\frac{4}{d}}. 
$$
Consequently, if $\|u_0\|_{L^2}<\|Q\|_{L^2}$, then $\| \sqrt{H}u(t) \|_{L^2}$ is bounded uniformly in $t$.  

\bigskip

\noindent
{\bf Case 3:}  $2+4/d<p<2^*$. 
Multiplying \eqref{eq:E-below-by-HGN} by $\|u_0\|_{L^2}^{q}$ with $q$ determined by 
$$
p(1-\theta) + q = \frac{q p\theta}{2}, \quad \text{namely } q=\frac{4d-2p(d-2)}{dp-2d-4}
$$
we obtain 
\begin{align}  \label{eq:E-M-lower}
E(u_0)\|u_0\|_{L^2}^{q}  &\ge \frac{1}{2} \| \sqrt{H}u(t) \|_{L^2}^2 \|u_0\|_{L^2}^{q}  - \frac{1}{pC_{\rm HGN}^p} \Big( \| \sqrt{H}u(t) \|_{L^2}^2 \|u_0 \|_{L^2}^{q} \Big)^{\frac{p\theta}{2}} \nn\\
& = f(\| \sqrt{H}u(t) \|_{L^2}^2 \|u_0\|_{L^2}^{q})
\end{align}
with  
$$
f(s):=\frac{1}{2}s -  \frac{1}{pC_{\rm HGN}^p} s^{\frac{p\theta}{2}}, \quad s\geq 0. 
$$
Following the argument of Holmer-Roudenko \cite{HolRou-07} (the function $f$ in \cite{HolRou-07} is defined slightly different from ours), we will use the fact that $f$ is strictly increasing in $[0,s_0]$ and strictly decreasing in $[s_0,\infty)$ where
$$
s_0 = \left(\frac{C_{\rm HGN}^p}{\theta}\right)^{\frac{2}{p\theta-2}}=  \| \sqrt{H} Q \|_{L^2}^2 \|Q\|_{L^2}^{q}.
$$

Now we prove that if  
\begin{align} \label{eq:EMK-u0-Q}
 E(u_0) \|u_0\|_{L^2}^{q} < E(Q) \|Q\|_{L^2}^{q}, \quad \| \sqrt{H} u_0\|_{L^2}^2 \|u_0\|_{L^2}^{q} < s_0 ,
\end{align}
then 
\begin{align} \label{eq:EMK-u0-Q-con}
 \| \sqrt{H} u(t)\|_{L^2}^2 \|u_0\|_{L^2}^{q} < s_0 
 \end{align}
for all $t>0$. First, \eqref{eq:EMK-u0-Q-con} holds at $t=0$ by the second condition in \eqref{eq:EMK-u0-Q}. Moreover, from \eqref{eq:E-M-lower} and the first condition in  \eqref{eq:EMK-u0-Q}  it follows that 
$$
 f(\| \sqrt{H}u(t) \|_{L^2}^2 \|u_0\|_{L^2}^{q}) \le E(u_0)\|u_0\|_{L^2}^{q} < E(Q) \|Q\|_{L^2}^{q} = f(s_0). 
$$
Therefore, since $f$ is strictly increasing in $[0,s_0]$,  by the continuity of $t\mapsto \| \sqrt{H}u(t) \|_{L^2}^2 \|u_0\|_{L^2}^{q}$ we conclude that $\| \sqrt{H}u(t) \|_{L^2}^2 \|u_0\|_{L^2}^{q}$ will never reach the maximum point $s_0$, namely  \eqref{eq:EMK-u0-Q-con} holds true for all $t$. Consequently,  \eqref{eq:EMK-u0-Q-con} implies that $ \| \sqrt{H} u(t)\|_{L^2}$ is bounded uniformly in $t$, which ensures the global existence of $u(t)$.

\subsection{Finite time blowup}

We will use the following result of Suzuki \cite[Subsection 3.1]{Suz-17}.  
\begin{lemma}[Virial identities] \label{lem:Virial} Let $d\ge 3$ and $2<p<2^*$. Let  $u\in \cQ$ be a solution of \eqref{eq:NLS} on $[0,T)$. If $|x|u_0\in L^2(\R^d)$,    
then $|x|u(t,x)\in L^2(\R^d)$ for all $t\in[0,T)$  and the function 
\begin{equation} \label{Gamma} 
\Gamma(t):=\int_{\R^d} |x|^2 |u(t,x)|^2dx
\end{equation}
satisfies the following identities for all $t \in [0,T)$, 
\begin{align}
 \Gamma'(t) &=4 \Im \int_{\R^d} \overline{x{u}(t,x)} \cdot\nabla u(t,x) dx, \label{eq:Virial-1}  \\ 
 \Gamma''(t)&=16E(u_0)+\frac{4+2d-dp}{p} \int_{\R^d} |u(t,x)|^{p}dx. \label{eq:Virial-2}
\end{align}
Moreover, for any $v\in \cQ$ and for any real-valued, radial function $\varphi$ such that $|x|\varphi v\in L^2(\R^d)$ we have
\begin{equation}  \label{eq:Virial-3} 
\left|  \Im \int_{\R^d} \overline{x \varphi v(t,x)} \cdot\nabla v(t,x) dx \right| \le \|x \varphi v\|_{L^2} (\| \sqrt{H} v\|_{L^2} + \|v\|_{L^2}). 
\end{equation}
\end{lemma}	

Note that \eqref{eq:Virial-3} ensures that the right side of \eqref{eq:Virial-1}  is finite as soon as $u(t)\in \cQ$.  

Now we prove part (iii) of Theorem \ref{thm:NLS-basic}. Assume that the solution $u(t)$ of  \eqref{eq:NLS} exists on $[0,T)$. We will show that, with $\Gamma(t)$ in \eqref{Gamma},  
\begin{align} \label{eq:Gamma''}
\Gamma''(t) \le - \lambda <0
\end{align}
for all $t\in [0,T)$, where $\lambda>0$ is a constant depending only on $u_0$. Note that by Taylor's expansion
$$
0\le \Gamma(t) = \Gamma(0) + t \Gamma'(0) + \frac{t^2}{2} \Gamma''(s_t) 
$$
(for some $s_t\in [0,t]$) the bound \eqref{eq:Gamma''} implies that 
$$
0\le \Gamma(t) \le \Gamma(0) + t \Gamma'(0)  - \frac{t^2}{2} \lambda
$$
for all $t\in [0,T)$. Since the latter bound cannot hold true for large $t$, we conclude that $u(t)$ must blow up at a finite time.

It remains to prove \eqref{eq:Gamma''}. If $E(u_0)<0$, then  \eqref{eq:Gamma''} follows immediately from the Virial identity \eqref{eq:Virial-2} and the fact that $4+2d-dp\le 0$:
$$
\Gamma''(t) = 16E(u_0)+\frac{4+2d-dp}{p} \int_{\R^d} |u(t,x)|^{p}dx  \le 16 E(u_0)<0. 
$$
Now instead of $E(u_0)<0$, we assume
$$
E(u_0) \|u_0\|_{L^2}^{q} < E(Q) \|Q\|_{L^2}^{q}, \quad \| \sqrt{H} u_0\|_{L^2}^2 \|u_0\|_{L^2}^{q} > \| \sqrt{H} Q\|_{L^2}^2 \|Q\|_{L^2}^{q}.
$$
We use again the argument of Holmer-Roudenko \cite{HolRou-07}. Note that 
$$
\| \sqrt{H}u_0 \|_{L^2}^2 \|u_0\|_{L^2}^{q} >  \| \sqrt{H} Q\|_{L^2}^2 \|Q\|_{L^2}^{q}=  s_0
$$
at time $t=0$, and 
$$
 f(\| \sqrt{H}u(t) \|_{L^2}^2 \|u_0\|_{L^2}^{q}) \le E(u_0)\|u_0\|_{L^2}^{q} < E(Q) \|Q\|_{L^2}^{q} = f(s_0)
$$
for all $t<T$ due to \eqref{eq:E-M-lower}. Since $f$ is strictly decreasing in $[s_0,\infty)$, by the continuity of $t\mapsto \| \sqrt{H}u(t) \|_{L^2}^2 \|u_0\|_{L^2}^{q}$ we conclude that
\begin{equation} \label{dicho}
\| \sqrt{H}u(t) \|_{L^2}^2 \|u_0\|_{L^2}^{q} > s_0
\end{equation}
for all $t<T$. Finally, multiplying the Virial identity \eqref{eq:Virial-2} with $\| u_0 \|_{L^2}^{q}$, then using \eqref{dicho} together with the facts that $p>2$ and $4-d(p-2)\le 0$ we obtain
\begin{align} \nonumber
\Gamma''(t) \| u_0 \|_{L^2}^{q} &= 4d(p-2) E(u_0)\| u_0 \|_{L^2}^q+ 2(4-d(p-2))\| \sqrt{H} u(t) \|_{L^2}^2 \| u_0 \|_{L^2}^q\\
&\le  4d(p-2)E(u_0)\| u_0 \|_{L^2}^q + 2(4-d(p-2))\| \sqrt{H} Q \|_{L^2}^2 \| Q \|_{L^2}^q \nn \\
&=  4d(p-2) \Big( E(u_0)\| u_0 \|_{L^2}^q -  E(Q)\| Q\|_{L^2}^q \Big) <0 .
\end{align}
Here in the last equality we have used
$$4d(p-2)E(Q)\| Q\|_{L^2}^q + 2(4-d(p-2))\| \sqrt{H} Q \|_{L^2}^2 \| Q \|_{L^2}^q =0.$$
Thus \eqref{eq:Gamma''} holds true, and hence $u(t)$ blows up at a finite time. This ends the proof of Theorem \ref{thm:NLS-basic}. 

\section{Minimal mass blowup solutions}

\subsection{Compactness of minimizing sequences} In this subsection we offer another, free-rearrangement proof of the existence of minimizers of the Hardy-Gagliardo-Nirenberg inequality \eqref{eq:HGN}. This proof implies an important consequence, that any (normalized) minimizing sequence of \eqref{eq:HGN}  is pre-compact (without the radial assumption). This will be a crucial  ingredient of our analysis of finite time blow-up solutions in Theorem \ref{thm:minimal-mass-solution}. For the completeness we will work on the general case $2<p<2^*$ (instead of focusing on the mass-critical case $p=2+4/d$). We have

\begin{theorem}[Compactness of minimizing sequences] \label{precompact} Let $d\ge 3$ and $2<p<2^*$. Let $\{u_n\}$ be a minimizing sequence for the Hardy-Gagliardo-Nirenberg inequality \eqref{eq:HGN}   such that 
$$
\liminf_{n\to \infty} \|u_n\|_{L^2}>0, \quad \limsup_{n\to \infty} \|u_n\|_{\cQ}<\infty. 
$$
Then there exist a subsequence of $\{u_n\}$ and constants $\lambda>0$, $z\in \mathbb{C}$ such that  
$$u_n(x)\to zQ(\lambda x) \quad \text{ strongly in } \cQ.$$  
\end{theorem}

\begin{proof} By dilations, we can assume that 
$$
\|u_n\|_{L^2}= \|\sqrt{H} u_n\|_{L^2}= 1, \quad  \|u_n\|_{L^p} \to C_{\rm HGN}^{-1}.
$$
Since $\{u_n\}$ is bounded in $\cQ$, thanks to \eqref{eq:Q-in-Hs} and Sobolev's embedding theorem we have, up to a subsequence when $n
\to \infty$, there exists $u_0\in \cQ$ such that
\begin{align*}
u_n &\wto u_0\quad \text{ weakly in $\cQ$ and $H^s(\R^d)$ for all $0<s<1$,}\\
u_n(x) &\to u_0(x) \quad \text{ for a.e. $x\in \R^d$,} \\
\1_{B(0,R)}u_n &\to \1_{B(0,R)} u_0 \quad \text{ strongly for all $2\le p<2^*$, for all $R>0$}.  
\end{align*}

We need to show that $\|u_0\|_{L^2}=1$. This will imply that $u_n\to u_0$ strongly in $L^2(\R^d)$, and hence $u_n\to u_0$ strongly in $L^q(\R^d)$ for all $2\le q<2^*$ by interpolation, which allows us to conclude by Fatou's lemma as in the above argument. 

We assume, for the sake of contradiction, that $\|u_0\|_{L^2}<1$. Then from the local convergence, we can find a sequence $R_n\to \infty$ such that when $n\to \infty$, 
$$
\int_{R_n \le |x|\le 2R_n} |u_n(x)|^2 dx \to 0, \quad \int_{|x|\le R_n} |u_n(x)|^2 dx \to m<1.
$$
Fix smooth functions $\chi, \eta:\R^d\to [0,1]$ such that 
$$
\chi^2+\eta^2=1, \quad \chi(x)=1 \text{ if } |x|\le 1, \quad \chi(x)=0  \text{ if } |x|\ge 2
$$
and define
$$
\chi_n (x) = \chi(x/R_n), \quad \eta_n(x) = \eta(x/R_n), \quad n \in \N.
$$
By the IMS formula we have the decomposition
\begin{align}\label{eq:localization-0}
\left\langle u_n, H u_n \right\rangle \nn &= \left\langle \chi_n u_n, H (\chi_n u_n) \right\rangle + \left\langle \eta_n u_n, -\Delta (\eta_n u_n) \right\rangle \nn\\
&\quad  - c_* \int_{\R^d} \frac{\eta_n^2}{|x|^2} |u_n(x)|^2  dx - \int_{\R^d} (|\nabla \chi_n(x)|^2 +|\nabla \eta_n(x)|^2) |u_n(x)|^2 dx \nn\\
&\geq \left\langle \chi_n u_n, H( \chi_n u_n) \right\rangle + \left\langle \eta_n u_n, -\Delta (\eta_n u_n) \right\rangle + o(1)_{n\to \infty}.
\end{align}
Then by H\"older's inequality, 
\begin{align} \label{eq:localization-Holder}
&\left\langle u_n, H u_n \right\rangle^{\theta} \left( \int_{\R^d} |u_n |^2 dx\right)^{1-\theta} \nn\\
& \geq \Big( \left\langle \chi_n u_n, H (\chi_n u_n) \right\rangle + \left\langle \eta_n u_n, -\Delta (\eta_n u_n) \right\rangle + o(1)_{n\to \infty} \Big)^{\theta} \times\nn\\
&\quad \times \left( \int_{\R^d} |\chi_n u_n |^2dx +  \int_{\R^d} |\eta_n u_n |^2dx \right)^{1-\theta} \nn\\
&\geq \left\langle \chi_n u_n, H (\chi_n u_n) \right\rangle^{\theta}  \left( \int_{\R^d} |\chi_n u_n |^2 dx \right)^{1-\theta} \nn\\
& \quad + \left\langle \eta_n u_n, -\Delta (\eta_n u_n) \right\rangle^{\theta} \left(   \int_{\R^d} |\eta_n u_n |^2dx \right)^{1-\theta} + o(1)_{n\to \infty}.
\end{align}
The first term on the right side of  \eqref{eq:localization-Holder} can be estimated using \eqref{eq:HGN}:
$$
\left\langle \chi_n u_n, H(\chi_n u_n) \right\rangle^{\theta}  \left( \int_{\R^d} |\chi_n u_n |^2dx \right)^{1-\theta} \ge C_{\rm HGN}^2 \|\chi_n u_n\|_{L^p}^2. 
$$
For the second term on the ride side of \eqref{eq:localization-Holder}, we use 
$$
\left\langle \eta_n u_n, -\Delta (\chi_n u_n) \right\rangle^{\theta}  \left( \int_{\R^d} |\eta_n u_n |^2dx \right)^{1-\theta} \ge C_{\rm GN}^2 \|\eta_n u_n\|_{L^p}^2
$$
where
$$
C_{\rm GN} =\inf_{\varphi\ne 0} \frac{ \|\sqrt{-\Delta}  \varphi\|_{L^2}^\theta \|\varphi \|_{L^2}^{1-\theta} }{\|\varphi\|_{L^p}}.
$$
Since it is well-known that $C_{\rm GN}$ has a minimizer (which is indeed unique up to translations and dilations), we must have $C_{\rm GN}>C_{\rm HGN}$. Denote 
$$
\frac{C_{\rm GN}}{C_{\rm HGN}} = 1+ \eps_0>1. 
$$
Thus in summary, from \eqref{eq:localization-Holder} we deduce that
\begin{align} \label{eq:localization-Holder-1}
\left\langle u_n, Hu_n \right\rangle^{\theta} \left( \int_{\R^d} |u_n |^2 dx\right)^{1-\theta} \geq C_{\rm HGN} \Big( \|\chi_n u_n\|_{L^p}^2 + (1+\eps_0) \|\eta_n u_n\|_{L^p}^2 \Big) + o(1)_{n\to \infty}.
\end{align}

Next, from 
$$
\int_{R_n \le |x|\le 2R_n} |u_n(x)|^2 dx \to 0
$$
we deduce that 
$$
\int_{R_n \le |x|\le 2R_n} |u_n(x)|^p dx \to 0.
$$
Combining with the elementary estimate $a^s+b^s \ge (a+b)^s$ with $a,b\geq 0$ and $s=2/p<1$, we obtain
\begin{align*}
 \|\chi_n u_n\|_{L^p}^2 +  \|\eta_n u_n\|_{L^p}^2 &= \left( \int_{\R^d} |\chi_n u_n|^p dx \right) ^{\frac{2}{p}} + \left( \int_{\R^d} |\eta_n u_n|^p dx \right)^{\frac{2}{p}} \\
 &\ge  \left( \int_{\R^d} |\chi_n u_n|^p dx+ \int_{\R^d} |\eta_n u_n|^p dx \right)^{\frac{2}{p}} \\
 &=  \left( \int_{\R^d} |u_n|^p dx + o(1)_{n\to \infty} \right)^{\frac{2}{p}} = \|u_n\|_{L^p}^2 + o(1)_{n\to \infty}. 
\end{align*}
Therefore, \eqref{eq:localization-Holder-1} reduces to 
\begin{align} \label{eq:localization-Holder-2}
\left\langle u_n, H u_n \right\rangle^{\theta} \left( \int_{\R^d} |u_n |^2 dx \right)^{1-\theta} \geq C_{\rm HGN} \Big( \|u_n\|_{L^p}^2 + \eps_0 \|\eta_n u_n\|_{L^p}^2 \Big) + o(1)_{n\to \infty}.
\end{align}
Since $\{u_n\}$ is a minimizing sequence for \eqref{eq:HGN}, \eqref{eq:localization-Holder-2} implies that
$$
\|\eta_n u_n\|_{L^p}^2 \to 0,
$$ 
and hence
\begin{equation} \label{eq:chin-un-full}
\|\chi_n u_n\|_{L^p}^2 = \|u_n\|_{L^p}^2 + o(1)_{n\to \infty}. 
\end{equation}

To conclude, we come back to use \eqref{eq:localization-0}:
\begin{align*}
\left\langle u_n, H u_n \right\rangle \ge \left\langle \chi_n u_n, H (\chi_n u_n) \right\rangle + o(1)_{n\to \infty}
\end{align*}
and the fact 
$$
\|\chi_n u_n\|_{L^2} \to m<1 =\int_{\R^d} |u_n|^2dx
$$
together with  \eqref{eq:HGN} and \eqref{eq:chin-un-full}. All the above estimates give
\begin{align*}
&\left\langle u_n, H u_n \right\rangle^{\theta} \left( \int_{\R^d} |u_n|^2 dx \right)^{1-\theta} \\
&\ge\left(\left\langle \chi_n u_n, H (\chi_n u_n) \right\rangle + o(1)_{n\to \infty} \right)^{\theta}  \left( \frac{\int_{\R^d} |\chi_n u_n|^2dx + o(1)_{n\to \infty}}{m}\right)^{1-\theta} \\
&\ge m^{\theta-1} C_{\rm HGN}^2 \|\chi_n u_n\|_{L^p}^2 + o(1)_{n\to \infty}\\
&\ge m^{\theta-1} C_{\rm HGN}^2 \Big(\| u_n\|_{L^p}^2 + o(1)_{n\to \infty}\Big) + o(1)_{n\to \infty}.
\end{align*}
Thus
$$
\frac{\left\langle u_n, H u_n \right\rangle^{\theta} \left( \int_{\R^d} |u_n|^2 dx  \right)^{1-\theta}}{\| u_n\|_{L^p}^2} \ge m^{\theta-1} C_{\rm HGN}^2 + o(1)_{n\to\infty}.
$$
Since $m<1$, this is a contradiction to the fact that $\{u_n\}$ is a minimizing sequence for  \eqref{eq:HGN}.

Thus we must have $u_n\to u_0$ strongly in $L^2(\R^d)$, and hence $u_n\to u_0$ in $L^p(\R^d)$ by interpolation. This allows us to conclude that $u_0$ is a minimizer, and that $u_n\to u_0$ strongly in $\cQ$. By Theorem \ref{thm:HGN}, we know that $u_0(x)=zQ(\lambda x)$ for some constants $\lambda>0$ and $z\in \mathbb{C}$. This ends the proof. 
\end{proof}

\subsection{Blowup profile at $t\to T$} Now we come back to Theorem \ref{thm:minimal-mass-solution}. We will focus on the mass-critical case $p=2+4/d$, where
the sharp constant in the Hardy-Gagliardo-Nirenberg inequality \eqref{eq:HGN} is
\begin{align} \label{eq:CHGN-mass-critical}
	C_{\rm HGN} = \left( \frac{d}{d+2} \right)^{\frac{d}{2(d+2)}} \|Q\|_{L^2}^{\frac{2}{d+2}}.
\end{align} 
We will prove 
\begin{lemma}[Mass concentration as $t\to T$] \label{lem:mass-concentration} Assume $p=2+4/d$. Let $u$ be a solution of \eqref{eq:NLS} in $[0,T)$ such that $\| u_0 \|_{L^2} = \|Q\|_{L^2}$ and $\lim_{t \nearrow T}\| u(t) \|_{\cQ}=\infty$.  Let $\{t_n \}$ be an increasing sequence converging to $T$ and denote $u_n(x)=u(t_n x)$. Then 
$$|u_n|^2 \to \|Q\|_{L^2}^2\delta_0 \quad \text{ in }(\mathcal{D}(\R^N))$$
in the sense of distributions. 
\end{lemma}
Here $(\mathcal{D}(\R^N))'$ denotes the space of distributions in $\R^N$ and $\delta_0$ is the Dirac measure concentrated at $x=0$.

\begin{proof}

 Denote 
$$ v_n(x) = \lambda_n^{\frac{d}{2}}u_n(\lambda_n x) \quad \text{with } \lambda_n =\frac{\| \sqrt{H}Q \|_{L^2}}{\| \sqrt{H}u_n \|_{L^2}}.
$$  
Then $\lim_{n \to \infty}\lambda_n = 0$ and 
\begin{align*}
	&\| v_n \|_{L^2} = \| u_n \|_{L^2} = \| u_0 \|_{L^2} = \|Q\|_{L^2}, \\
	&\| \sqrt{H}v_n \|_{L^2} = \lambda_n \| \sqrt{H}u_n \|_{L^2}=\| \sqrt{H}Q \|_{L^2}, \\
	&\| v_n \|_{L^{2+\frac{4}{d}}}^{2+\frac{4}{d}} = \lambda_n^2 \| u_n \|_{L^{2+\frac{4}{d}}}^{2+\frac{4}{d}}. 
\end{align*}
Using the above identities and the energy conservation, we obtain
\begin{align*}
	E(v_n) &=\frac{1}{2}\| \sqrt{H}v_n \|_{L^2}^2 - \frac{d}{2(d+2)}\| v_n \|_{L^{2+\frac{4}{d}}}^{2+\frac{4}{d}} \\
	&=\lambda_n^2 \left(  \frac{1}{2}\| \sqrt{H}u_n \|_{L^2}^2 - \frac{d}{2(d+2)}\| u_n \|_{L^{2+\frac{4}{d}}}^{2+\frac{4}{d}} \right) \\
	&=\lambda_n^2  E(u_n) =\lambda_n^2  E(u_0) \to 0. 
\end{align*}
Consequently, 
	$$ 
	\lim_{n \to \infty} \| v_n \|_{L^{2+\frac{4}{d}}}^{2+\frac{4}{d}} = \lim_{n \to \infty}  \left(  \frac{d+2}{d} \right) \| \sqrt{H}v_n \|_{L^2} = \left(  \frac{d+2}{d} \right) \| \sqrt{H}Q \|_{L^2},
	$$
	and hence  
\begin{align*}
	\lim_{n \to \infty}\frac{ \| \sqrt{H} v_n\|_{L^2}^{\frac{d}{d+2}} \| v_n \|_{L^2}^{\frac{2}{d+2}}}{ \| v_n \|_{L^{2+\frac{4}{d}}}} = \left( \frac{d}{d+2} \right)^{\frac{d}{2(d+2)}} \|Q\|_{L^2}^{\frac{2}{d+2}} = C_{\rm HGN}.
\end{align*} 
It means that $\{ v_n \}$ is a minimizing sequence for \eqref{eq:HGN}. 
	
By Theorem \ref{precompact}, there exist a subsequence, still denoted by $\{ v_n \}$ and constants $\lambda>0$, $z\in \mathbb{C}$ such that $v_n(x) \to z Q(\lambda x)$ strongly in $\cQ$. Since 
$$
	\|v_n\|_{L^2} = \|Q\|_{L^2}, \quad  \| \sqrt{H}v_n  \|_{L^2} = \| \sqrt{H}Q \|_{L^2}
$$
we know that 
$$\lambda=|z|=1.$$
In particular, we obtain 
$$|v_n|^2 \to |Q|^2\quad \text{in  }L^1(\R^d).$$

Next, for any $\phi \in C_c^{\infty}(\R^d)$, we can write 
\begin{equation*}
\begin{aligned}
\langle |u_n|^2,\phi\rangle &=\int_{\R^d} |u_n(y)|^2\phi(y)dy\\
&=\int_{\R^d}|v_n(x)|^2 \phi(\lambda_n x) dx\\
&= \int_{\R^d} (|v_n(x)|^2-|Q(x)|^2) \phi(\lambda_n x)dx + \int_{\R^d} |Q(x)|^2\phi(0)dx \\
&+ \int_{\R^d}|Q(x)|^2 (\phi(\lambda_n x)-\phi(0))dx.
\end{aligned}
\end{equation*}
Since $\|v_n\|_{L^2}=\|Q\|_{L^2}$ we obtain
\begin{align}\label{eq:est-1}
|\langle u_n^2,\phi \rangle - \|Q\|_{L^2}^2 \phi(0)| &\leq  \|\phi\|_{L^{\infty}}\int_{\R^d}||v_n(x)|^2-|Q(x)|^2|dx \nn \\
&\quad +\int_{\R^d} |Q(x)|^2 |\phi(\lambda_n x)-\phi(0)|dx.
\end{align}
Since $|v_n|^2 \to Q^2$ in $L^1(\R^d)$, the first term on the right-hand side of \eqref{eq:est-1} tends to zero as $n \to \infty$.  Moreover, since $\lambda_n \to 0$ and $\phi \in C_c^{\infty}(\R^d)$  it follows that $\phi(\lambda_n x) \to \phi(0)$ as $n \to \infty$ for all $x \in \R^d$.  By invoking dominated convergence theorem, we derive that the second term on the right-hand side of \eqref{eq:est-1} tends to zero as $n \to \infty$. As a consequence,
$|u_n|^2 \to \|Q\|_{L^2}^2\delta_0$ in $(\mathcal{D}(\R^N))'$. \medskip
\end{proof}

\subsection{Virial identities}  To conclude the proof, we will use the Virial identities in Lemma \ref{lem:Virial}. Strictly speaking, the results in Lemma \ref{lem:Virial} holds under the condition $|x|u_0\in L^2(\R^d)$, which is not available here. However, this condition can be relaxed in the mass-critical case.

\begin{lemma}[Virial identity in the mass-critical case] \label{lem:Virial-mass}Let $u$ be a solution of \eqref{eq:NLS} in $[0,T)$ such that $\| u_0 \|_{L^2} = \|Q\|_{L^2}$ and $\lim_{t \nearrow T}\| u(t) \|_{\cQ}=\infty$.  Then for all $t\in [0,T)$ we have $|x|u(t)\in L^2(\R^d)$ and 
\begin{equation} \label{eq:Virial-mass}
\int_{\R^d} |x|^2 |u(t,x)|^2 d x=8E(u_0)(T-t)^2. 
\end{equation}
\end{lemma}

First, by using the a-priori estimate \eqref{eq:Virial-3} in Lemma \ref{lem:Virial} we can easily adapt a result of Banica \cite[Lemma 2.1]{Ban-04} to our case. 

\begin{lemma} \label{est:hab1} Assume $p=2+4/d$. Let $u \in \cQ$ such that $\| u \|_{L^2}=\|Q\|_{L^2}$. Then for any $\theta \in C_0^\infty(\R^d)$, there holds
$$
\left| \int_{\R^d} \nabla \theta \cdot \Im (\bar u \nabla u)dx  \right| \leq \sqrt{2E(u)} \left( \int_{\R^d}|\nabla \theta|^2 |u|^2 dx \right)^{\frac{1}{2}}.
$$
\end{lemma}

Now we provide

\begin{proof}[Proof of Lemma \ref{lem:Virial-mass}] Let $\phi \in C_0^{\infty}(\R^d)$ be a radial, non-negative function such that $\phi(x)=|x|^2$ for $|x| \leq 1.$ Since $\phi$ is non-negative radially symmetric, we can write $\phi(x)=\zeta(r)$ with $r=|x|$.

By Taylor's Theorem for all $r, \rho \in \R$, there exists $\tilde r \in [r,r+\rho]$ such that
\begin{equation}\label{eq:est-2}
\begin{aligned}
0\leq \zeta(r+\rho)=\zeta(r)+\rho \zeta'(r)+\frac{\rho^2}{2}\zeta''(\tilde r) \leq \zeta(r)+\rho \zeta'(r)+c_1 \rho^2,
\end{aligned}
\end{equation}
where $c_1=1+\max_{r \in \R}\frac{|\zeta''(r)|}{2}$.
We note that, the right-hand side of \eqref{eq:est-2} is a second degree polynomial in $\rho$, hence $|\zeta'(r)|^2-4c_1\zeta(r) \leq 0$, which implies  $|\zeta'(r)|^2 \leq C \zeta(r)$ with ${C=4c_1}$. Therefore, we have
\begin{equation}\label{eq:est-phi}
|\nabla \phi(x)|^2 \leq C\phi(x) \quad \text{for } x \in \R^d.
\end{equation}

For $R>0,$ define $\psi_R(x)=R^2\psi(\frac{x}{R})$ and 
$$ \Gamma_{R}(t):=\int_{\R}\psi_R(x)|u(t,x)|^2dx \quad \forall t \in [0,T).
$$
An easy computation yields
\begin{equation*}
\begin{aligned}
\Gamma'_{R}(t)=2\int_{\R^d} \nabla \psi_R \,  \Im (\bar{u} \nabla u) dx \quad \forall t \in [0,T).
\end{aligned}
\end{equation*}
Since $\| u (t) \|_{L^2}=\|Q\|_{L^2}$, it follows from Lemma \ref{est:hab1} and  \eqref{eq:est-phi} that
\begin{equation*}
\begin{aligned}
|\Gamma_{R}'(t)| &\leq 2 \sqrt{2E(u)} \bigg(\int_{\R^d} |\nabla \psi_R|^2 |u|^2dx\bigg)^{\frac{1}{2}}\\
&\leq 2 \sqrt{2E(u_0)} \bigg(\int_{\R^d} C^2  \psi_R |u|^2dx\bigg)^{\frac{1}{2}}\\
&\leq C \sqrt{E(u_0)} \sqrt{\Gamma_{R}(t)}.
\end{aligned}
\end{equation*}
Integrating between fixed $t \in [0,T)$ and $t_n$ we obtain,
\begin{equation}\label{eq:Gamma_R-est}
|\sqrt{\Gamma_{R}(t)}-\sqrt{\Gamma_{R}(t_n)}| \leq C|t-t_n|.
\end{equation}
By the mass concentration in Lemma \ref{lem:mass-concentration} we derive
\begin{equation}\label{eq:Gamma_t_n}
\Gamma_{R}(t_n)=\int_{\R^d}\psi_R(x) |u_n(x)|^2dx \to \|Q\|_{L^2}^2\psi_R(0)=0.
\end{equation}
By letting $n \to \infty$ in \eqref{eq:Gamma_t_n} and employing \eqref{eq:Gamma_R-est}, we deduce that $\Gamma_{R}(t) \leq C(T-t)^2$. This implies
\[ \int_{\R^d} \psi_R(x) |u(t,x)|^2 dx \leq C(T-t)^2. 
\]
Letting $R \to \infty$ and using monotone convergence theorem lead to
$ |x|u(t) \in L^2(\R^d)$ for all $t \in [0,T)$ and  
$$\Gamma(t):=\int_{\R^d}|x|^2|u(t,x)|^2dx \leq C(T-t)^2.$$
This allows to extend $\Gamma(t)$ by continuity at $t=T$ by setting $\Gamma(T)=0$ and consequently $\Gamma'(T)=0$. We obtain the \eqref{eq:Virial-2}, which  in the mass-critical case $p=2+4/d$ boils down to 
$$\Gamma''(t)=16E(u_0).$$
Combining with $\Gamma(T)=\Gamma'(T)=0$ we find that 
\begin{equation*}
\Gamma(t)=8E(u_0)(T-t)^2.
\end{equation*}
Consequently, 
\begin{align*} \Gamma(0) &= \int_{\R^d} |x|^2 |u_0|^2dx=8E(u_0)T, \\
\Gamma'(0) &=4 \Im \int_{\R^d}  \overline{xu_0(x)} \cdot \nabla u_0(x)dx=-16E(u_0)T.
\end{align*}
\end{proof}

\subsection{Conclusion} For any $T>0$, $\lambda>0$ and $\gamma \in \R$, define
 \begin{equation} \label{S}
\bS_{T,\lambda,\gamma}(t,x): = e^{i \gamma}e^{i\frac{\lambda^2}{T-t}}e^{-i\frac{|x|^2}{4(T-t)}}\left( \frac{\lambda}{T-t} \right)^{\frac{d}{2}} Q\left( \frac{\lambda x}{T-t} \right) \quad x\in \R^d, t \in [0,T).
\end{equation}
It is straightforward to check that  for any $T>0$, $\lambda>0$ and $\gamma \in \R$, the function $\bS_{T,\lambda,\gamma}$ is a minimal-mass solution of \eqref{eq:NLS} which blows up at finite time $T>0$.

Next we conclude by using the strategy of Hmidi-Keraani \cite{HmiKer-05}. We observe that for any $u\in \cQ$, for any real-valued function $\theta \in C_0^{\infty}(\R^d,\R)$ and $s \in \R$, there holds
\begin{equation}\label{eq:uei}
E(ue^{is\theta})=E(u)+s \, \Im \int_{\R^d}  \overline{u} \nabla \theta \cdot   \nabla u dx+\frac{s^2}{2}\int_{\R^d} |\nabla \theta|^2 |u|^2 dx.
\end{equation}
We can take $s=1/(2T)$ and choose $\theta(x)$ approaching $|x|^2/2$ (using appropriate cut-off functions). This gives 
\begin{equation*}
\begin{aligned}
E(u_0e^{\frac{i|x|^2}{4T}})&=E(u_0)+\frac{1}{2T} \Im \int_{\R^d} \overline{x u_0} \cdot \nabla u_0 dx+\frac{1}{8T^2}\int_{\R^d} |x|^2 |u_0|^2dx\\
&=E(u_0)-\frac{4E(u_0)}{2T}+\frac{1}{8T^2}(8E(u_0)T^2)=0.
\end{aligned}
\end{equation*}
Here we have used the Virial identity \eqref{eq:Virial-mass} in the last equality. Thus the function 
$$v_0 = u_0 e^{\frac{i|x|^2}{4T}}$$
satisfies that $\|v_0\|_{L^2} = \|Q\|_{L^2}$ and $E(v_0)=0$. Hence $v_0$ is a minimizer for the Hardy-Gagliardo-Nirenberg inequality \eqref{thm:HGN}. By Theorem \ref{thm:HGN}, there exist $\lambda_1>0, \gamma_1 \in \R$ such that 
$$u_0(x) e^{\frac{i|x|^2}{4T}}=e^{i \gamma_1} \lambda_1^{\frac{d}{2}}Q(\lambda_1 x), \quad \forall x \in \R^d,$$
which is equivalent to
$$
u_0(x)=e^{i\gamma_1}e^{-\frac{i|x|^2}{4T}}\lambda_1^{\frac{d}{2}}Q(\lambda_1 x) \quad x \in \R^d.
$$
Define $\lambda_0=\lambda_1 T>0$,  $\gamma_0=\gamma_1-\lambda_1^2T$, then we can write
$$ u_0(x)=e^{i \gamma_0}e^{i \frac{\lambda_0^2}{T}}e^{-i\frac{|x|^2}{4T}} \left( \frac{ \lambda_0}{T}\right)^{\frac{d}{2}}Q\left(\frac{\lambda_0 x}{T}\right)=\bS_{T,\lambda_0,\gamma_0}(0,x) \quad x \in \R^d.
$$ 
By the uniqueness, we conclude that
$u(t,x)=\bS_{T,\lambda_0,\gamma_0}(t,x)$ for all $t \in [0,T)$. 
This complete the proof  of Theorem \ref{thm:minimal-mass-solution}.

\section{Extension to the case $c|x|^{-2}$ with $c<c_*$} \label{sec:ext}

Instead of \eqref{eq:NLS}, we may also consider the NLS with non-critical inverse square potential 
\begin{equation} \label{eq:NLS-Q-c} 
\left\{  \begin{aligned}
i \partial_t u(t,x) &= (-\Delta -c|x|^{-2}) u(t,x) - |u(t,x)|^{p-2}u(t,x), \quad x\in \R^d, t>0, \\
u(0,x) &=u_0(x), \quad x\in \R^d,
\end{aligned} \right. \end{equation}
with $d\geq 3, 2<p<2^*=2d/(d-2)$ and 
$$ 
c< c_*=\frac{(d-2)^2}{4}.
$$

All the results in Theorems \ref{thm:ground-state}, \ref{thm:HGN}, \ref{thm:NLS-basic}, \ref{thm:minimal-mass-solution} hold true in this subcritical case. In fact, the proof in this case is often simpler since now the quadratic form domain of $-\Delta -c|x|^{-2}$ is simply $H^1(\R^d)$. 

Let us quickly explain how to adapt our proof to this case. We will focus on Theorem \ref{thm:ground-state}, namely the existence and uniqueness of the positive radial solution to 
\begin{align} \label{eq:Q-cccc}
-\Delta Q - c|x|^{-2} Q - Q^{p-1}+Q=0.
\end{align}
(our result is new for the case $c<c_*$ as well). The existence part is easy and we only consider the uniqueness part. The main difference when $c<c_*$ is that we have the following analogue of the asymptotic formula \eqref{eq:Q-0} 
\begin{align} \label{eq:Q-00}
\lim_{|x| \to 0}|x|^{\kappa}Q(x) \in (0,\infty)
\end{align}
with
\begin{equation} \label{kappa}
\kappa:= \frac{d-2}{2} - \sqrt{\left(\frac{d-2}{2}\right)^2 -c}.
\end{equation}
Then in the proof of the uniqueness of $Q$, when $c<c_*$ we will replace \eqref{eq:v-Q-def} by  
$$
v(r)=r^{\kappa}Q(r),
$$
and proceed exactly the same as in the critical case to get the desired result. 

Let us explain \eqref{eq:Q-00} in more detail. If $Q\in H^1(\R^d)$ is a positive radial solution of \eqref{eq:Q-cccc}, then using the polar coordinate $r=|x|$ we find that $Q$ satisfies
$$ \frac{d^2}{dr^2}  Q + \frac{N-1}{r} \frac{d}{dr} Q + \frac{c}{r^2} Q - Q + Q^{p-1} = 0 \quad \text{in } (0,\infty).
$$
Put
$$ \ell:=\frac{d-2-2\kappa}{d-2}
$$
and 
$$ W(s):=r^{\kappa}Q(r) \quad \text{with } s=r^\ell,
$$
then $W$ satisfies
$$\frac{d^2}{ds^2}  W +\frac{d-1}{s} \frac{d}{ds} W - \frac{1}{\ell^2}s^{\frac{2(1-\ell)}{\ell}}W + \frac{1}{\ell^2}s^{\frac{2(1-\ell)-\kappa(p-2)}{\ell}}W^{p-1} = 0 \quad \text{in } (0,\infty).
$$
Equivalently, we can write
$$ \Delta W - \frac{1}{\ell^2}|x|^{\frac{2(1-\ell)}{\ell}}W + \frac{1}{\ell^2}|x|^{\frac{2(1-\ell)-\kappa(p-2)}{\ell}}W^{p-1} = 0 \quad \text{in } \R^d \setminus \{0\}.
$$
Since $\ell \in (0,1)$ and $p<2^*=\frac{2d}{d-2}$, we have
$$\lim_{|x| \to 0} |x|^{\frac{2(1-\ell)}{\ell}} = \lim_{|x| \to 0} |x|^{\frac{2(1-\ell)-\kappa(p-2)}{\ell}} =0.$$
Therefore, by a similar argument as in \cite[Proof of Theorem 1.2]{TraZog-15}, we deduce that $W(0)$ is well defined as a positive number. Thus \eqref{eq:Q-00} holds true.

\bibliographystyle{siam}

\end{document}